\documentclass[12pt]{elsarticle}
\usepackage{graphicx}
\usepackage{amsthm}
\usepackage{amsmath}
\usepackage{amssymb}
\usepackage{mathrsfs}
\usepackage{multirow}
\usepackage{color}
\usepackage{indentfirst}
\usepackage{geometry}
\usepackage{CJK}
\usepackage{makecell}
\usepackage{graphicx,subcaption}
\usepackage{epstopdf}
\usepackage{booktabs}
\usepackage{overpic}
\usepackage{tikz}
\usepackage{hyperref}  

\newtheorem{remark}{Remark}
\newtheorem{theorem}{Theorem}[section]
\newtheorem{lemma}{Lemma}[section]
\numberwithin{equation}{section} 
\newtheorem{proposition}{Proposition} 
\newtheorem{problem}{Problem}%

\usepackage{geometry}
\biboptions{square,comma}

\journal{Journal of Computational Physics}
\begin{document}
\begin{frontmatter}

\title{Improved error estimates of a new splitting scheme\\ for charged-particle dynamics in strong magnetic field\\ with  maximal ordering}

\author[mymainaddress]{Mengting Hu}
\ead{humting@stu.xjtu.edu.cn}

\author[mysecondaryaddress]{Jiyong Li}
\cortext[mycorrespondingauthor]{Corresponding author.}
\ead{ljyong406@163.com}

\author[mymainaddress]{Bin Wang\corref{mycorrespondingauthor}}
\ead{wangbinmaths@xjtu.edu.cn}

\address[mymainaddress]{School of Mathematics and Statistics, Xi’an Jiaotong University, Xi’an, Shanxi 710049, China}
\address[mysecondaryaddress]{School of Mathematical Sciences, Hebei Key Laboratory of Computational Mathematics and Applications, Hebei International Joint Research Center for Mathematics and Interdisciplinary Science, Hebei Normal University, Shijiazhuang, 050024, China}


\begin{abstract}
This paper introduces a novel second-order splitting scheme for charged-particle dynamics in strong magnetic fields characterized by the maximal ordering. The proposed scheme is explicit and symmetric, which respectively ensure the efficiency of the algorithm and its long-term near-conservation of energy.  We rigorously prove that the scheme achieves improved error bounds for both the position and the velocity component parallel to the magnetic field, yielding a uniform second-order error bound under specific strong-field regimes. Numerical experiments confirm the optimal convergence rates and the long-term energy near conservation of the method.
\end{abstract}


\end{frontmatter}


\section{Introduction}\label{sec1}
In this paper, we investigate the numerical solution of the charged-particle dynamics (CPD) in a strong magnetic field, governed by the following system of differential equations:
\begin{equation}
	\begin{split}
		\dot{x}(t) & = v(t),  \\
		\dot{v}(t) & = v(t) \times \mathcal{B}( x(t)) + E(x(t)), \quad 0 < t \leq T,  \\
		x(0) & = x_{0}, \quad v(0) = v_{0},
	\end{split}
	\label{1.1}
\end{equation}
where the unknown functions $x(t):[0,T]\rightarrow\mathbb{R}^{3}$ and $v(t):[0,T]\rightarrow\mathbb{R}^{3}$ represent the position and velocity of the particle, respectively, with given initial conditions $x_{0}, v_{0} \in \mathbb{R}^{3}$. The expression $\mathcal{B}(x) \in \mathbb{R}^3$ denotes the magnetic field and the electric field $E(x) = -\nabla U(x)$ is derived from a scalar potential $U(x)$. The dimensionless parameter $0<\varepsilon \ll 1 $ is inversely proportional to the strength of the magnetic field.
We consider the CPD under the maximal ordering scaling (MOS) \cite{Poss2018}:
\begin{equation}
	\mathcal{B}(x) = \frac{B(\varepsilon^q x)}{\varepsilon},
	\label{BX}
\end{equation}
where $q \in [1,2]$ and \(B=(b_{1},b_{2},b_{3})^{\intercal}\in\mathbb{R}^{3}\) is a magnetic field. This scaling is predicated on two key physical assumptions. First, the magnetic field is significantly stronger than the electric field, which is quantified by the condition
$
\varepsilon \sim \frac{|E|}{c|\mathcal{B}|} \ll 1.
$
Second, the magnetic field varies slowly in space, with its spatial variation being of order $\varepsilon^q$. This is expressed by the inequality
$\varepsilon^q \sim \rho \frac{|\nabla \mathcal{B}|}{\mathcal{B}} \ll 1,$
where $\rho$ denotes the characteristic length of the particle dynamics, such as the gyro-radius.
A fundamental property of the system \eqref{1.1} is the conservation of energy. Along the solution, the Hamiltonian $H(t)$ of the system
\begin{equation}
	H\left(x(t),v(t)\right) := \frac{1}{2}\left|v(t)\right|^{2} + U(x(t)) \equiv H\left(x(0),v(0)\right), \quad t \geq 0,
	\label{1.2}
\end{equation}
is conserved.

Charged-particle dynamics, as a core topic in classical electrodynamics, has a long-standing and deeply researched history in the physics community \cite{ref18,ref19,ref20,ref21,ref22}.
In fusion devices like tokamaks, strong magnetic fields are employed to confine high-temperature plasmas and control the trajectories of charged particles within them, thereby enabling controlled fusion reactions.
This significant application context has greatly driven advancements in the detailed modeling and large-scale numerical simulation of CPD in strong magnetic field environments.
Consequently, \eqref{1.1} frequently arises as a central problem in the particle discretization of various kinetic models \cite{ref24,ref25,ref26,ref27,ref28,ref29,ref30,ref31,ref32,ref33}.

Extensive research has been devoted to the case of $\varepsilon = 1$ in \eqref{1.1}. The Boris method \cite{ref34}, originally proposed in 1970, has since had its mathematical properties systematically analyzed \cite{ref36}. Furthermore, a wide array of structure-preserving algorithms has been developed for this regime, including symplectic algorithms \cite{ref41}, K-symplectic algorithms \cite{ref38}, variational symplectic algorithms \cite{ref48}, explicit time-symmetric multistep algorithms \cite{ref37}, Poisson integrators \cite{ref42}, volume-preserving algorithms \cite{ref43}, and various energy-preserving algorithms \cite{ref44,ref45}.

There has also been extensive research on numerical algorithms for \eqref{1.1} under strong magnetic fields (i.e., $0 < \varepsilon \ll 1$).	
Some structure-preserving methods were analyzed in \cite{ref9,ref6}.
A numerical discretization using variational integrators was  proved to maintain approximate conservation properties over long time scales in  \cite{ref1}.    However, the error of these algorithms   usually depends on $\varepsilon$, and the error increases as $\varepsilon$ decreases.  
To overcome this severe step-size constraint, uniformly accurate algorithms have been proposed. In the context of strong magnetic fields, such algorithms have been designed for the Vlasov equation \cite{UA2,UA3,UA1} by combining them with Particle-in-Cell (PIC) discretization. These schemes achieve errors and time-step constraints entirely independent of $\varepsilon$. In \cite{UA4}, a class of uniformly accurate algorithms was derived for \eqref{1.1}.
It is important to note that these methods rely on two-scale reformulations, which expand the dimensionality of the problem and consequently incur higher computational costs compared to traditional approaches.

Seeking to balance the efficiency of traditional algorithms with the favorable error bounds, filter algorithms have been investigated. A filtered Boris algorithm was proposed in \cite{filter1} for \eqref{1.1}, followed by the study of filtered variational integrators in \cite{filter2}. Although these methods demonstrate improved accuracy and long-term approximate conservation of invariants, they remain implicit, necessitating the solution of nonlinear systems at each step and thus reducing computational efficiency.
To obtain uniform error bounds with explicit traditional methods, \cite{ref11} developed several first-order algorithms via splitting techniques for \eqref{1.1}. However, although extending this approach to second-order accuracy is nontrivial, the resulting second-order schemes fail to retain uniform error bounds, even for  specialized magnetic fields such as the uniform strong case \cite{ref12}.

In this work, we introduce a novel splitting strategy that achieves second-order uniform error bounds for a class of strong magnetic fields. The proof relies on transforming the original system into a long-time problem via time scaling. Leveraging the skew-symmetry of the magnetic field, we identify a propagator that generates a periodic flow. By carefully analyzing error propagation over each period for a long time, we deduce the improved error bounds.
The main advantages of the proposed scheme are summarized as follows:
\begin{itemize}
	\item Based on Strang splitting, the scheme is explicit and straightforward to implement, avoiding any iterative processes and ensuring high computational efficiency.
	\item   The scheme is symmetric, thereby demonstrating excellent long-term approximate energy conservation properties.
\item  For \eqref{1.1}  with $q=2$ and uniform strong  magnetic field, our scheme maintains a second-order uniform error bound for both the position and the velocity component parallel to the magnetic field.
	For \(1 < q < 2\), the error bound of the our scheme is \(\mathcal{O}(\varepsilon^{q-2} h^2)\) ($h$ denotes the time step size), which outperforms that of traditional second-order splitting schemes. Even in the case of \(q = 1\), our scheme still achieves better error performance compared to traditional schemes, as will be demonstrated in the numerical experiments.
	
\end{itemize}

The outline of the paper is as follows. Section \ref{sec2} introduces the new explicit second-order splitting scheme and presents its properties, demonstrated by a series of numerical experiments. In Section \ref{sec3}, we provide a rigorous convergence analysis. Section \ref{sec4} gives some concluding remarks.
Throughout this paper, we use the notation
\begin{equation*}
	A \lesssim B
\end{equation*}
to indicate that  $A \leq CB$, where $ C > 0$  is a generic constant independent of the time stepsize $h$, the number of time steps $n$ and $\varepsilon$.

\section{New splitting method}\label{sec2}
In this section, we present the formulation of the method, study its properties and test its performance by some numerical experiments.

\subsection{Formulation and properties of the method}\label{subsec2.1}
Denote \(h>0\) as the time step size and \(t_{n}=nh\) for \(n\in\mathbb{N}\). 
By introducing the magnetic field term $\frac{B(\varepsilon^{q} x_{0})}{\varepsilon} v(t)$, we can split \eqref{1.1} into the following two subflows:
\begin{equation} 
	\frac{d}{dt}\begin{pmatrix}x(t)\\ v(t)\end{pmatrix}=
	\mathcal{S}_{x_{0}}\begin{pmatrix}x(t)\\ v(t)\end{pmatrix}
	+ \mathcal{T}_{x_{0}}\begin{pmatrix}x(t)\\ v(t)\end{pmatrix},
	\label{2.1}
\end{equation}
where
\begin{equation*}
	\begin{gathered} 
		\mathcal{S}_{x_{0}}\begin{pmatrix}x(t)\\ v(t)\end{pmatrix}
		= \begin{pmatrix}
			v(t) \\
			\displaystyle\frac{\widehat{B}(\varepsilon^{q} x_{0})}{\varepsilon}v(t)
		\end{pmatrix}, \\
		\mathcal{T}_{x_{0}}\begin{pmatrix}x(t)\\ v(t)\end{pmatrix}
		= \begin{pmatrix}
			0 \\
			\displaystyle\frac{\widehat{B}(\varepsilon^{q} x(t))-\widehat{B}(\varepsilon^{q} x_{0})}{\varepsilon}v(t)
			+ E\bigl(x(t)\bigr)
		\end{pmatrix},
	\end{gathered}
\end{equation*}
with the skew symmetric matrix 
\begin{equation*}
	\widehat{B}(x)=\begin{pmatrix}
		0 & b_{3}(x) & -b_{2}(x) \\ 
		-b_{3}(x) & 0 & b_{1}(x) \\ 
		b_{2}(x) & -b_{1}(x) & 0
	\end{pmatrix}
\end{equation*} 
and the magnetic field \(B=(b_{1},b_{2},b_{3})^{\intercal}\in\mathbb{R}^{3}\).
We denote the exact flows of the subsystems corresponding to \(\mathcal{S}_{x_{0}}\) and \(\mathcal{T}_{x_{0}}\) by \(\varphi_{t}^{\mathcal{S}_{x_{0}}}\) and \(\varphi_{t}^{\mathcal{T}_{x_{0}}}\), respectively. 
For the first flow, we get its exact propagator
\begin{equation}
	\begin{pmatrix}x(t_{n+1}) \\ v(t_{n+1})\end{pmatrix}
	= \varphi_{h}^{\mathcal{S}_{x_{0}}}\begin{pmatrix}x(t_{n}) \\ v(t_{n})\end{pmatrix}
	:= \begin{pmatrix}
		x(t_{n}) + h\varphi_{1}\bigl(\frac{h}{\varepsilon}\widehat{B}(\varepsilon^{q} x_{0})\bigr)v(t_{n}) \\
		\mathrm{e}^{\frac{h}{\varepsilon}\widehat{B}(\varepsilon^{q} x_{0})}v(t_{n})
	\end{pmatrix},
	\label{2.3}
\end{equation}
where \(\varphi_{1}(s)=({\rm e}^{s}-1)/s\). For the second flow, the exact propagator is
\begin{gather}
	\begin{pmatrix}x(t_{n+1})\\ v(t_{n+1})\end{pmatrix}
	= \varphi_{h}^{\mathcal{T}_{x_{0}}}\begin{pmatrix}x(t_{n}) \\ v(t_{n})\end{pmatrix} \notag \\
	:= \begin{pmatrix}
		x(t_{n}) \\
		\mathrm{e}^{\frac{h}{\varepsilon}\bigl(\widehat{B}
			(\varepsilon^{q} x(t_{n}))-\widehat{B}(\varepsilon^{q} x_{0})\bigr)}v(t_{n})
		+ h\varphi_{1}\bigl(\frac{h}{\varepsilon}\bigl(\widehat{B}
		(\varepsilon^{q} x(t_{n}))-\widehat{B}(\varepsilon^{q} x_{0})\bigr)\bigr)E(x(t_{n}))
	\end{pmatrix}.
	\label{2.4}
\end{gather}

The numerical method is obtained through composition.

\noindent\textbf{Algorithm 2.1.} 
Given \(x^0 = x_0\) and \(v^0 = v_0\), and denoting the numerical solutions as \(x^{n}\approx x(t_{n})\) and \(v^{n}\approx v(t_{n})\), the Strang splitting scheme  
\begin{equation*}
	\begin{pmatrix}x^{n+1}\\ v^{n+1}\end{pmatrix} = \varphi_{h/2}^{S_{x_{0}}}\circ\varphi_{h}^{\mathcal{T}_{x_{0}}}
	\circ\varphi_{h/2}^{S_{x_{0}}}
	\begin{pmatrix}x^{n}\\ v^{n}\end{pmatrix}
\end{equation*}
for solving \eqref{1.1} in total reads as
\begin{align}
	\begin{split}
		\displaystyle x^{n+1}=& x^{n}+\frac{h}{2}\varphi_{1}\biggl(\frac{h\widehat{B}_{0}}{2\varepsilon}\biggr)
		\Bigl(I+\mathrm{e}^{h\frac{\widehat{B}_{\bar{z}_{n}}-\widehat{B}_{0}}{\varepsilon}}
		\mathrm{e}^{\frac{h\widehat{B}_{0}}{2\varepsilon}}\Bigr)v^{n}  +\frac{h^{2}}{2}\varphi_{1}\biggl(\frac{h\widehat{B}_{0}}{2\varepsilon}\biggr)
		\varphi_{1}\biggl(h\frac{\widehat{B}_{\bar{z}_{n}}-\widehat{B}_{0}}{\varepsilon}\biggr)
		E(\bar{z}_{n}), \\
		\displaystyle v^{n+1}=& \mathrm{e}^{\frac{h\widehat{B}_{0}}{2\varepsilon}}
		\mathrm{e}^{h\frac{\widehat{B}_{\bar{z}_{n}}-\widehat{B}_{0}}{\varepsilon}}
		\mathrm{e}^{\frac{h\widehat{B}_{0}}{2\varepsilon}}v^{n}  +h\mathrm{e}^{\frac{h\widehat{B}_{0}}{2\varepsilon}}
		\varphi_{1}\biggl(h\frac{\widehat{B}_{\bar{z}_{n}}-\widehat{B}_{0}}{\varepsilon}\biggr)
		E(\bar{z}_{n}), \quad 0\leq n < \frac{T}{h},
	\end{split}
	\label{2.5}
\end{align}
with the notations
\begin{equation}
	\bar{z}_{n}=x^{n}+\frac{1}{2}h\varphi_{1}\biggl(\frac{h\widehat{B}_{0}}{2\varepsilon}\biggr)v^{n},\quad
	\widehat{B}_{0}=\widehat{B}(\varepsilon^{q} x_{0}),\quad
	\widehat{B}_{\bar{z}_{n}}=\widehat{B}(\varepsilon^{q} \bar{z}_{n}).
	\label{2.6}
\end{equation}
We refer to the scheme \eqref{2.5} as S2-new and can easily prove that it is a symmetric scheme.

\begin{proposition}
	The S2-new scheme \eqref{2.5} with \eqref{2.6} is time symmetric.
\end{proposition}

\begin{proof}
	Swapping $(n,h)$ with $(n+1,-h)$ in \eqref{2.5} leads to the following result
	\begin{align}
		\displaystyle x^{n}= & x^{n+1}-\frac{h}{2}\varphi_{1}\biggl(-\frac{h\widehat{B}_{0}}{2\varepsilon}\biggr)
		\Bigl(I+\mathrm{e}^{-h\frac{\widehat{B}_{\widehat{z}_{n+1}}-\widehat{B}_{0}}{\varepsilon}}
		\mathrm{e}^{\frac{-h\widehat{B}_{0}}{2\varepsilon}}\Bigr)v^{n+1}  \notag \\ \displaystyle & + \frac{h^{2}}{2}\varphi_{1}\biggl(\frac{-h\widehat{B}_{0}}{2\varepsilon}\biggr)
		\varphi_{1}\biggl(-h\frac{\widehat{B}_{\widehat{z}_{n+1}}-\widehat{B}_{0}}{\varepsilon}\biggr)
		E(\widehat{z}_{n+1}), \label{eq1}\\
		\displaystyle 
		v^{n}= & \mathrm{e}^{\frac{-h\widehat{B}_{0}}{2\varepsilon}}
		\mathrm{e}^{-h\frac{\widehat{B}_{\widehat{z}_{n+1}}-\widehat{B}_{0}}{\varepsilon}}
		\mathrm{e}^{\frac{-h\widehat{B}_{0}}{2\varepsilon}}v^{n+1}  - h\mathrm{e}^{\frac{-h\widehat{B}_{0}}{2\varepsilon}}
		\varphi_{1}\biggl(-h\frac{\widehat{B}_{\widehat{z}_{n+1}}-\widehat{B}_{0}}{\varepsilon}\biggr)
		E(\widehat{z}_{n+1}), \label{eq2}
	\end{align}
	where
	\begin{equation*}
		\widehat{z}_{n+1} = x^{n+1} - \frac{1}{2}h\varphi_{1} \biggl(\frac{-h\widehat{B}_{0}}{2\varepsilon}\biggr)v^{n+1},\quad
		\widehat{B}_{\widehat{z}_{n+1}}=\widehat{B}(\varepsilon^{q} \widehat{z}_{n+1}).
	\end{equation*}
	From \eqref{2.5} and the definition of \(\varphi_{1}(s)=({\rm e}^{s}-1)/s\), we can first derive
	\begin{equation*}
		x^{n}+\frac{1}{2}h\varphi_{1}\biggl(\frac{h\widehat{B}_{0}}{2\varepsilon}\biggr)v^{n} = 
		x^{n+1} - \frac{1}{2}h\varphi_{1} \biggl(\frac{-h\widehat{B}_{0}}{2\varepsilon}\biggr)v^{n+1},
	\end{equation*}
	i.e., $\bar{z}_{n} = \widehat{z}_{n+1}.$
	Therefore, \eqref{eq2} can be rewritten as
	\begin{align}
		\displaystyle v^{n+1} = & \mathrm{e}^{\frac{h\widehat{B}_{0}}{2\varepsilon}}
		\mathrm{e}^{h\frac{\widehat{B}_{\widehat{z}_{n+1}}-\widehat{B}_{0}}{\varepsilon}}
		\mathrm{e}^{\frac{h\widehat{B}_{0}}{2\varepsilon}}v^{n}
		+ h \mathrm{e}^{\frac{h\widehat{B}_{0}}{2\varepsilon}}
		\mathrm{e}^{h\frac{\widehat{B}_{\widehat{z}_{n+1}}-\widehat{B}_{0}}{\varepsilon}}
		\varphi_{1}\biggl(-h\frac{\widehat{B}_{\widehat{z}_{n+1}}-\widehat{B}_{0}}{\varepsilon}\biggr)
		E(\widehat{z}_{n+1})  \notag \\
		= & \mathrm{e}^{\frac{h\widehat{B}_{0}}{2\varepsilon}}
		\mathrm{e}^{h\frac{\widehat{B}_{\bar{z}_{n}}-\widehat{B}_{0}}{\varepsilon}}
		\mathrm{e}^{\frac{h\widehat{B}_{0}}{2\varepsilon}}v^{n} 
		\displaystyle + h\mathrm{e}^{\frac{h\widehat{B}_{0}}{2\varepsilon}}
		\varphi_{1}\biggl(h\frac{\widehat{B}_{\bar{z}_{n}}-\widehat{B}_{0}}{\varepsilon}\biggr)
		E(\bar{z}_{n}). \label{eq4}
	\end{align}
	Similarly, due to $\bar{z}_{n} = \widehat{z}_{n+1}$ and \eqref{eq4}, \eqref{eq1} can be rewritten as
	\begin{align}
		x^{n+1} = & x^{n} +
		\frac{h}{2}\varphi_{1}\biggl(-\frac{h\widehat{B}_{0}}{2\varepsilon}\biggr)
		\Bigl(I+\mathrm{e}^{-h\frac{\widehat{B}_{\widehat{z}_{n+1}}-\widehat{B}_{0}}{\varepsilon}}
		\mathrm{e}^{\frac{-h\widehat{B}_{0}}{2\varepsilon}}\Bigr)v^{n+1} \notag \\
		&  - \frac{h^{2}}{2}\varphi_{1}\biggl(\frac{-h\widehat{B}_{0}}{2\varepsilon}\biggr)
		\varphi_{1}\biggl(-h\frac{\widehat{B}_{\widehat{z}_{n+1}}-\widehat{B}_{0}}{\varepsilon}\biggr)
		E(\widehat{z}_{n+1})   \notag \\
		= & x^{n} +
		\frac{h}{2}\varphi_{1}\biggl(-\frac{h\widehat{B}_{0}}{2\varepsilon}\biggr)
		\Bigl(I+\mathrm{e}^{-h\frac{\widehat{B}_{\bar{z}_{n}}-\widehat{B}_{0}}{\varepsilon}}
		\mathrm{e}^{\frac{-h\widehat{B}_{0}}{2\varepsilon}}\Bigr)  \notag \\
		& \,\, \Biggl[ \mathrm{e}^{\frac{h\widehat{B}_{0}}{2\varepsilon}}
		\mathrm{e}^{h\frac{\widehat{B}_{\bar{z}_{n}}-\widehat{B}_{0}}{\varepsilon}}
		\mathrm{e}^{\frac{h\widehat{B}_{0}}{2\varepsilon}}v^{n} 
		\displaystyle + h\mathrm{e}^{\frac{h\widehat{B}_{0}}{2\varepsilon}}
		\varphi_{1}\biggl(h\frac{\widehat{B}_{\bar{z}_{n}}-\widehat{B}_{0}}{\varepsilon}\biggr) E(\bar{z}_{n}) \Biggr] \notag \\
		& - \frac{h^{2}}{2}\varphi_{1}\biggl(\frac{-h\widehat{B}_{0}}{2\varepsilon}\biggr)
		\varphi_{1}\biggl(-h\frac{\widehat{B}_{\bar{z}_{n}}-\widehat{B}_{0}}{\varepsilon}\biggr)
		E(\bar{z}_{n}) \notag \\
		= & x^{n} +
		\frac{h}{2}\varphi_{1}\biggl(-\frac{h\widehat{B}_{0}}{2\varepsilon}\biggr)
		\Bigg[ \mathrm{e}^{\frac{h\widehat{B}_{0}}{2\varepsilon}}
		\mathrm{e}^{h\frac{\widehat{B}_{\bar{z}_{n}}-\widehat{B}_{0}}{\varepsilon}}
		\mathrm{e}^{\frac{h\widehat{B}_{0}}{2\varepsilon}}v^{n} +
		\mathrm{e}^{\frac{h\widehat{B}_{0}}{2\varepsilon}}v^{n}
		+ h\mathrm{e}^{\frac{h\widehat{B}_{0}}{2\varepsilon}}
		\varphi_{1}\biggl(h\frac{\widehat{B}_{\bar{z}_{n}}-\widehat{B}_{0}}{\varepsilon}\biggr) E(\bar{z}_{n}) \Biggr] \notag \\
		= & x^{n}+\frac{h}{2}\varphi_{1}
		\biggl(\frac{h\widehat{B}_{0}}{2\varepsilon}\biggr)
		\Bigl(I+\mathrm{e}^{h\frac{\widehat{B}_{\bar{z}_{n}}-\widehat{B}_{0}}{\varepsilon}}
		\mathrm{e}^{\frac{h\widehat{B}_{0}}{2\varepsilon}}\Bigr)v^{n}
		+\frac{h^{2}}{2}\varphi_{1}\biggl(\frac{h\widehat{B}_{0}}{2\varepsilon}\biggr)
		\varphi_{1}\biggl(h\frac{\widehat{B}_{\bar{z}_{n}}-\widehat{B}_{0}}{\varepsilon}\biggr)
		E(\bar{z}_{n}).
		\label{eq5}
	\end{align}
	\eqref{eq4} and \eqref{eq5} demonstrate that the S2-new scheme \eqref{2.5} is time symmetric. The proof is complete. 
\end{proof}

The convergence result for the proposed scheme is stated as follows. A rigorous proof will be presented in Section \ref{sec3}.

\begin{theorem}\label{th2.1}
	Assume that \( E(\cdot), B(\cdot) \in C^2(\mathbb{R}^3) \), and let \(T_0 > 0\) be the single time period in the periodic flow generated by \(e^{t \widehat{B}(0)}\). 
	Denote \( x^n \) and \( v^n \) be the numerical solutions obtained by applying the S2-new scheme to system \eqref{1.1} on the time interval \([0,T]\), where \( T > 0 \) is fixed.
	Then there exists a positive constant \( N_0 \) independent of \( \varepsilon \), such that for every integer \( N \ge N_0 \), by setting the time step size \( h = \frac{T_0}{N}\varepsilon \), the following error bounds hold:
	\begin{equation}
		|x^n - x(t_n)| \lesssim \varepsilon^{q-2} h^2, 
		\quad |v^n_{\parallel} - v_{\parallel}(t_n)| \lesssim \varepsilon^{q-2} h^2,
		\quad 1 \leq q \leq 2,  
		\quad 0 \leq n \leq \frac{T}{h}. \label{2.7}
	\end{equation}
	If the magnetic field is uniform, i.e., $B\equiv B_0$, the S2-new scheme has the following uniform error bounds
	\begin{equation}
		|x^n - x(t_n)| \lesssim   h^2, 
		\quad |v^n_{\parallel} - v_{\parallel}(t_n)| \lesssim   h^2, 
		\quad 0 \leq n \leq \frac{T}{h}. \label{2.7B0}
	\end{equation}
	The above \( v_{\parallel} \) represents the component of velocity \( v \) parallel to the magnetic field \( B \)
	\begin{equation*}
		v_{\parallel}(t) := \frac{B(\varepsilon^q x(t))}{|B(\varepsilon^q x(t))|} \left( \frac{B(\varepsilon^q x(t))}{|B(\varepsilon^q x(t))|} \cdot v(t) \right),  
	\end{equation*}
	and similarly for $v^n$ as
	\begin{equation*}
		v_{\parallel}^n := \frac{B(\varepsilon^q x^n)}{|B(\varepsilon^q x^n)|} \left( \frac{B(\varepsilon^q x^n)}{|B(\varepsilon^q x^n)|} \cdot v^n \right).
	\end{equation*}
\end{theorem}

\begin{remark}
	The primary contribution of this paper is the redesign of the traditional Strang splitting method, resulting in a novel splitting scheme \eqref{2.5}  that achieves the improved error bounds \eqref{2.7}-\eqref{2.7B0}. 	For equation \eqref{1.1} under the strong magnetic field \eqref{BX}, previous work \cite{ref12} analyzed two splitting methods, S2-AVF and S2-VP \cite{ref11}, demonstrating that both methods yield an error bound of $\mathcal{O}(\varepsilon^{-1} h^2)$.
	In contrast, our S2-new scheme achieves an improved error bound of $\mathcal{O}(\varepsilon^{q-2} h^2)$ for $q \in [1,2]$.
	Furthermore, even in the case of $q=1$, the numerical experiments presented in the following subsection demonstrate that the S2-new scheme outperforms the S2-VP scheme.
\end{remark}

\subsection{Numerical experiments}\label{subsec2.2}
To illustrate the numerical performance of S2-new \eqref{2.5}  for different values of $q \in [1,2]$, we shall conduct numerical experiments in this subsection. 
The CPD is solved numerically up to the final time $T = t_{N} = 1$, and the relative errors are calculated as follows:
\begin{align}
	& errx := \frac{|x^N - x(t_N)|}{|x(t_N)|}, 
	\quad errv_{\parallel} := \frac{|v_{\parallel}^N - v_{\parallel}(t_N)|}{|v_{\parallel}(t_N)|},  \label{2.15}  \\
	& \hspace{1.5cm} error := errx + errv_{\parallel}.  
	\label{2.16}
\end{align}
The energy error is as follows:
\begin{equation}
	e_{H} := \frac{|H(x^{n},v^{n}) - H(x^{0},v^{0})|}{|H(x^{0},v^{0})|}, 
	\quad 0 \leq n \leq N. 
	\label{2.17}
\end{equation}
And the reference solutions are obtained using MATLAB's “ode45”.

For comparison, we considered a splitting method in the following from \cite{ref11}:
\begin{equation*}
	\Phi_{t}^{1}:\quad \frac{d}{dt}\begin{pmatrix}
		x \\ 
		v
	\end{pmatrix}
	=
	\begin{pmatrix}
		v \\
		0
	\end{pmatrix},
	\qquad
	\Phi_{t}^{2}:\quad \frac{d}{dt}\begin{pmatrix}
		x \\ 
		v
	\end{pmatrix}
	=
	\begin{pmatrix}
		0 \\
		\dfrac{1}{\varepsilon}v\times B(\varepsilon^{q} x)+E(x)
	\end{pmatrix},
\end{equation*}
where both subflows have exact integrators \(\varphi_{t}^{1},\varphi_{t}^{2}\). 
Denoting the following Strang splitting method
\begin{equation*}
	\begin{pmatrix}x^{n+1}\\ v^{n+1}\end{pmatrix}
	= \varphi_{h/2}^{1}\circ\varphi_{h}^{2}\circ\varphi_{h/2}^{1}
	\begin{pmatrix}x^{n}\\ v^{n}\end{pmatrix}
\end{equation*}
as S2-VP, we obtain
\begin{align}
	\begin{split}
		x^{n+1} & = x^{n}+\frac{h}{2} v^{n} + \frac{h}{2} v^{n+1},  \quad n \geq 0,\\
		v^{n+1} & = \mathrm{e}^{\frac{h}{\varepsilon} \widehat{B}(\varepsilon^{q} (x^{n} + \frac{h}{2} v^{n}))} v^{n} +h \varphi_{1}\biggl(\frac{h}{\varepsilon} 
		\widehat{B}(\varepsilon^{q} (x^{n} + \frac{h}{2} v^{n}))   \biggr) 
		E(x^{n} + \frac{h}{2} v^{n}).
	\end{split}
	\label{2.19}
\end{align}

\begin{problem}\label{pro1}
	(\textbf{Uniform strong magnetic field.})	The first illustrative numerical experiment is devoted to the charged-particle motion in a uniform strong magnetic field
	\begin{equation*}
		\frac{1}{\varepsilon}B=\frac{1}{\varepsilon}
		(1,0,0.5)^{\intercal},
	\end{equation*}
	and the electric field \(E(x)=-\nabla U(x)\) with the potential \(U(x)=\dfrac{1}{\sqrt{x_{1}^{2}+x_{2}^{2}+x_{3}^{2}}}\). 
	We choose the initial values as \(x(0)=(0,1,0.1)^{\intercal}\) and \(v(0)=(0.09,0.05,0.2)^{\intercal}\).
\end{problem}

The relative errors \eqref{2.15}-\eqref{2.16} and the energy errors \eqref{2.17} of S2-new \eqref{2.5} and S2-VP \eqref{2.19} are shown in Figs.~\ref{S2_C0}-\ref{E_C0}. From these results, we can obviously observe the following behaviour:
\begin{enumerate}
	\item Two splitting schemes S2-new and S2-VP demonstrate a second-order error bound of $h$ for varying values of $\varepsilon \in (0,1)$ in the position $x$ and parallel velocity $v_{\parallel}$. 
	
	\item The  scheme S2-new shows that the uniform error bound is independent of $\varepsilon$, which verifies the theoretical result in \eqref{2.7B0} of Theorem \ref{th2.1}.
	In contrast, the  scheme S2-VP shows  a dependence on   $\varepsilon^{-1}$. 
	\item  Fig.~\ref{E_C0} shows that both splitting schemes nearly preserve the energy over long times.
\end{enumerate}
 
\begin{figure}[t!] 
	\centering
	\begin{subfigure}{0.328\textwidth}
		\centering
		\includegraphics[width=\linewidth]{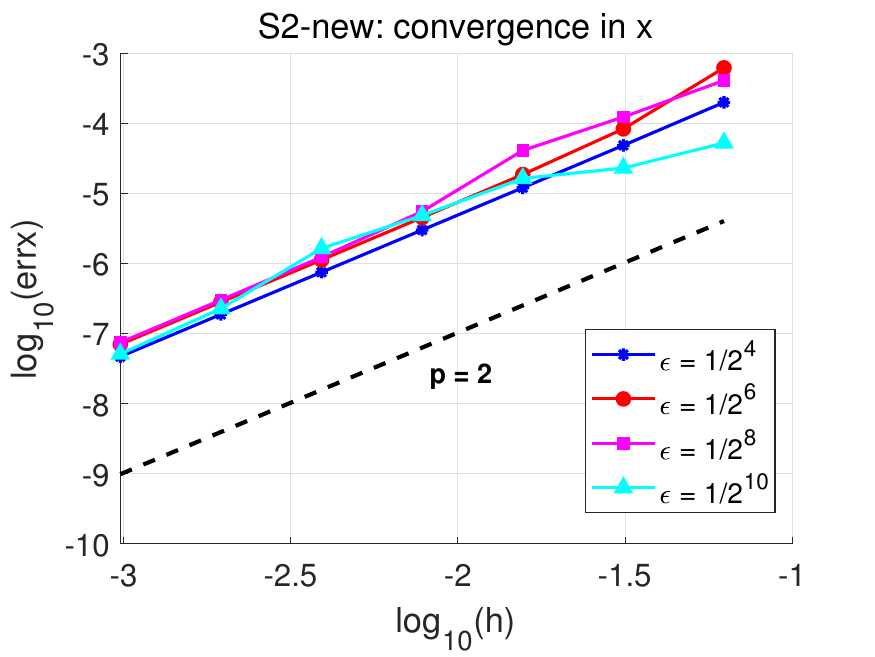}
	\end{subfigure}
	\begin{subfigure}{0.328\textwidth}
		\centering
		\includegraphics[width=\linewidth]{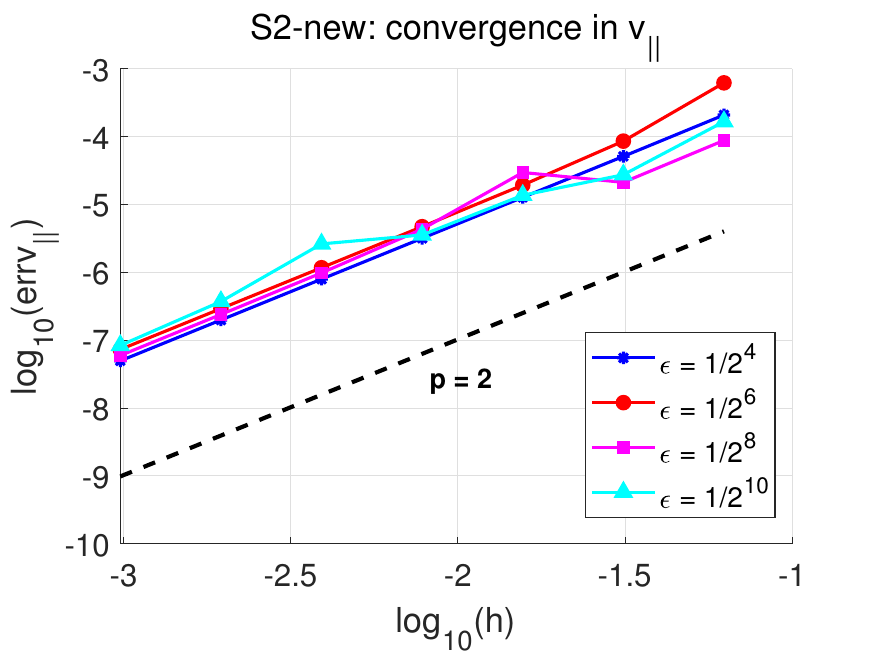}  
	\end{subfigure}
	\begin{subfigure}{0.328\textwidth}
		\centering
		\includegraphics[width=\linewidth]{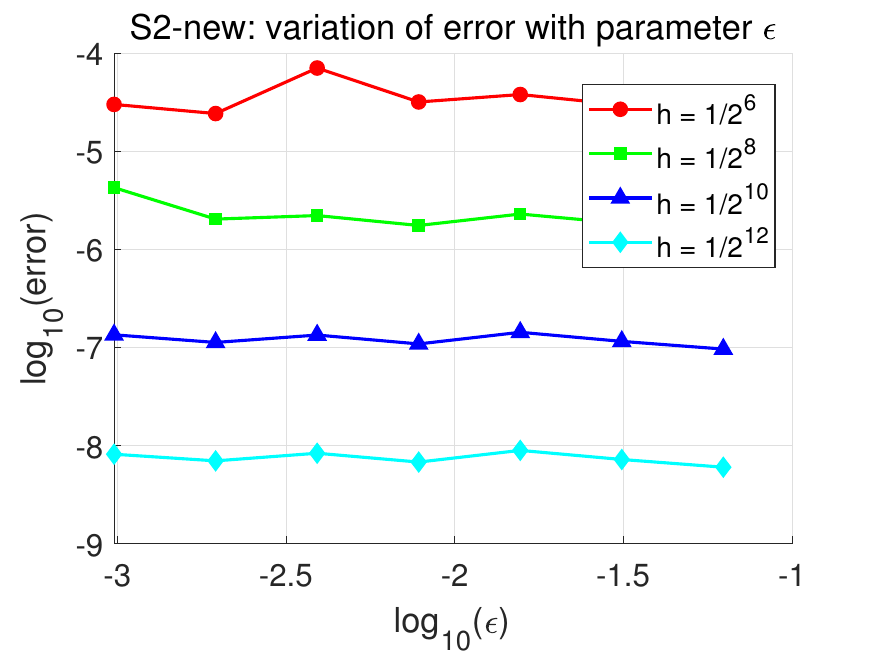}
	\end{subfigure}
	\caption{Problem \ref{pro1}. The error \eqref{2.15} (left and middle) of S2-new with $h = 1/2^{k}$ for $k=4,...,10$ under different $\varepsilon$ in a uniform strong magnetic field. The error \eqref{2.16} (right) of S2-new with $\varepsilon = 1/2^{k}$ for $k=4,...,10$ under different $h$ in a uniform strong magnetic field.}
	\label{S2_C0}
\end{figure}
\begin{figure}[t!] 
	\centering
	\begin{subfigure}{0.328\textwidth}
		\centering
		\includegraphics[width=\linewidth]{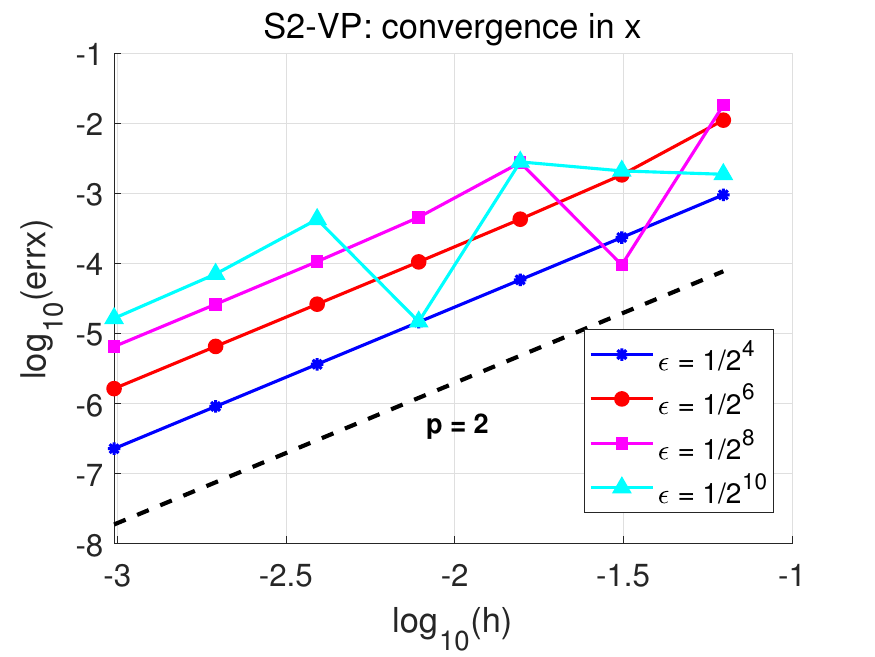}
	\end{subfigure}
	\begin{subfigure}{0.328\textwidth}
		\centering
		\includegraphics[width=\linewidth]{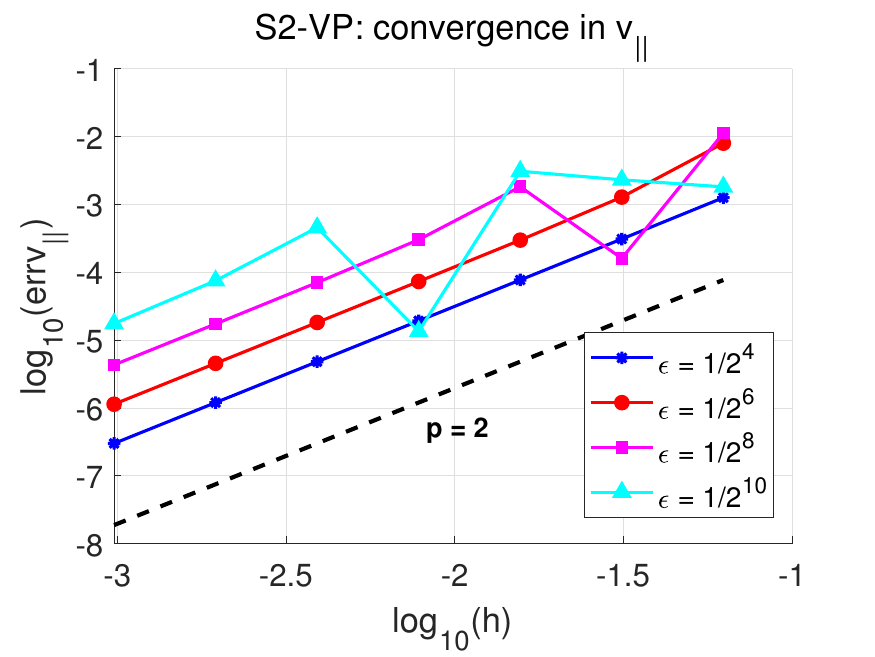}  
	\end{subfigure}
	\begin{subfigure}{0.328\textwidth}
		\centering
		\includegraphics[width=\linewidth]{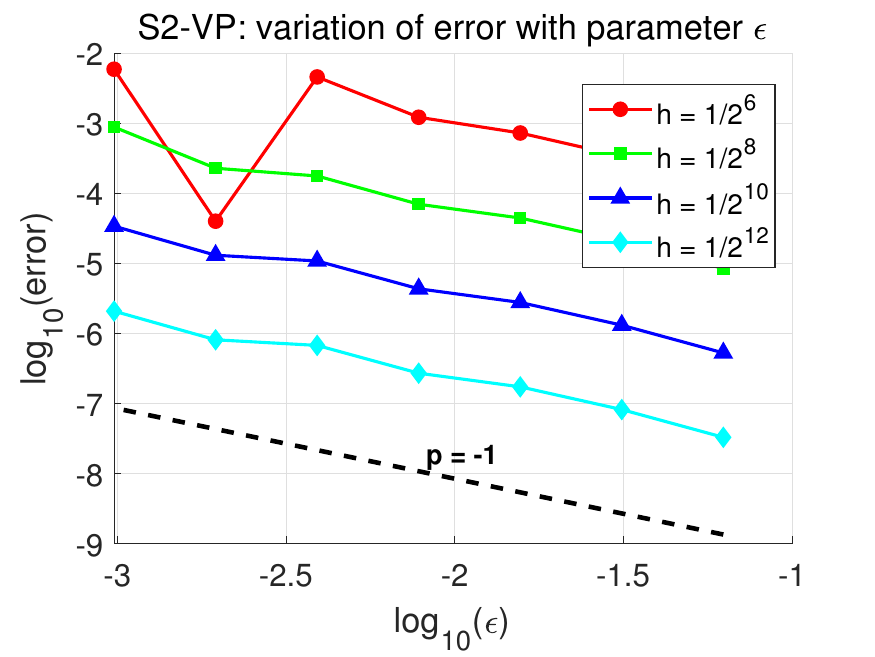}
	\end{subfigure}
	\caption{Problem \ref{pro1}. The error \eqref{2.15} (left and middle) of S2-VP with $h = 1/2^{k}$ for $k=4,...,10$ under different $\varepsilon$ in a uniform strong magnetic field. The error \eqref{2.16} (right) of S2-VP with $\varepsilon = 1/2^{k}$ for $k=4,...,10$ under different $h$ in a uniform strong magnetic field.}
	\label{VP_C0}
\end{figure}

\begin{figure}[t!] 
	\centering
	\begin{subfigure}{0.328\textwidth}
		\centering
		\includegraphics[width=\linewidth]{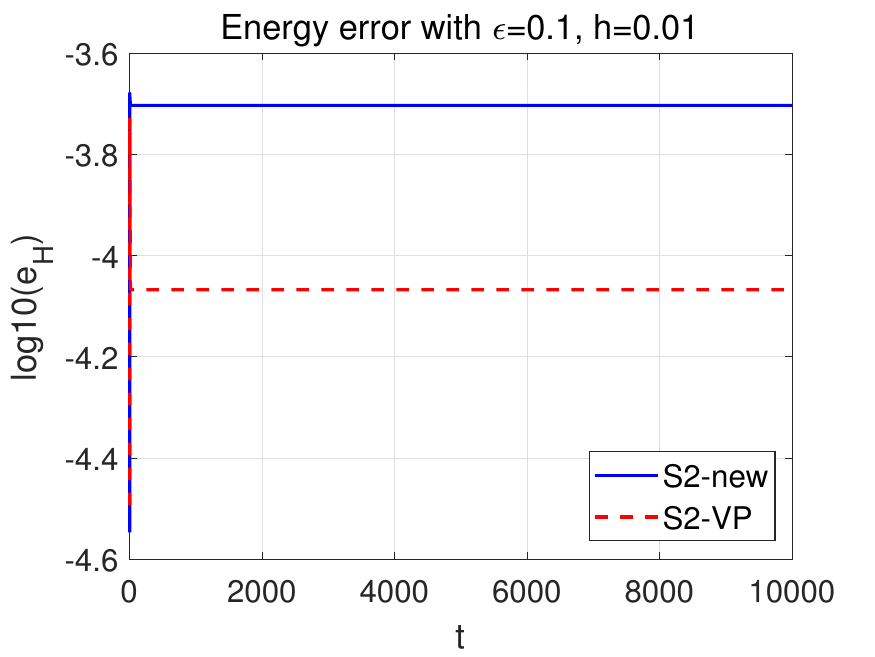}
	\end{subfigure}
	\begin{subfigure}{0.328\textwidth}
		\centering
		\includegraphics[width=\linewidth]{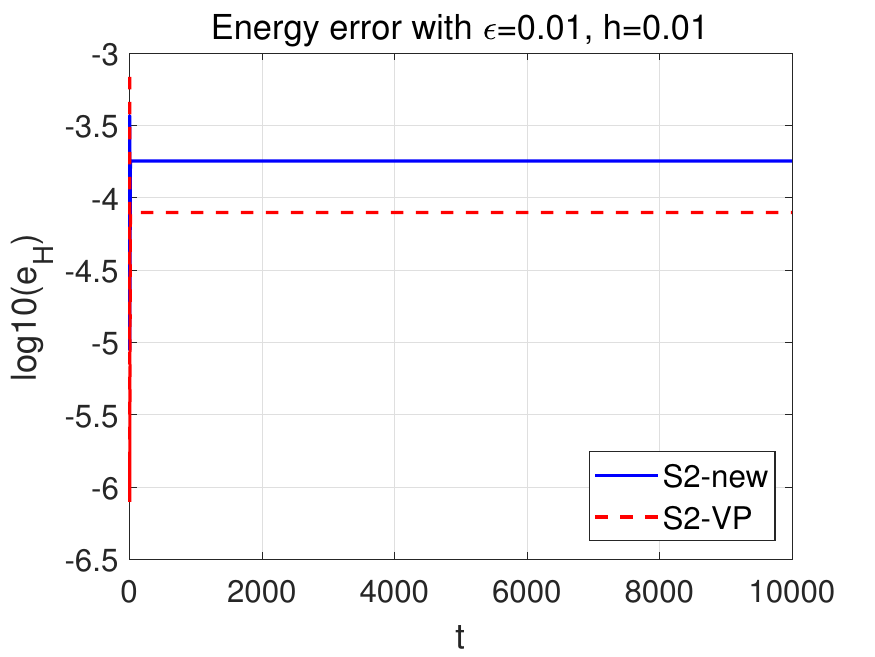}  
	\end{subfigure}
	\begin{subfigure}{0.328\textwidth}
		\centering
		\includegraphics[width=\linewidth]{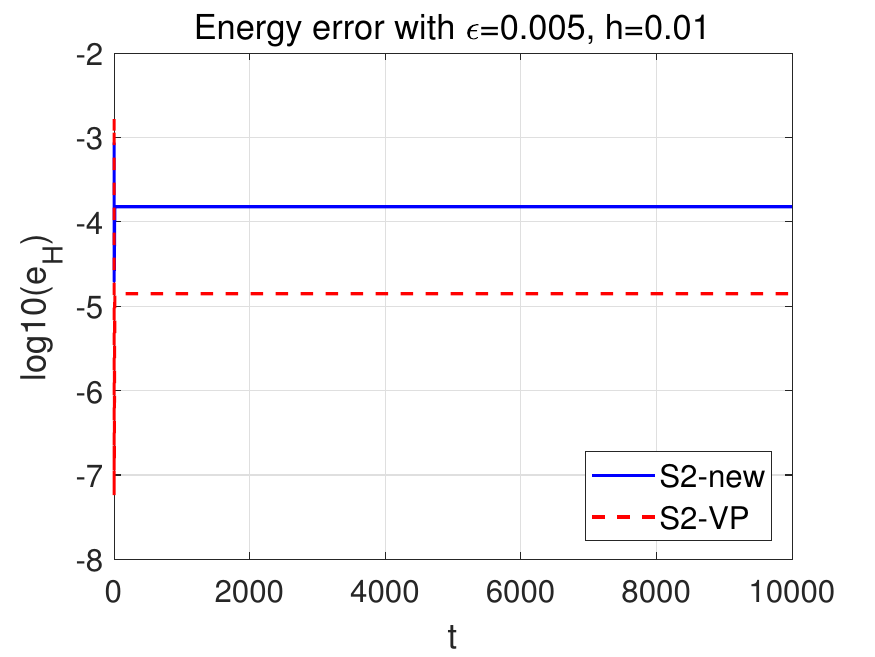}  
	\end{subfigure}
	\caption{Problem \ref{pro1}. Evolution of the energy error \eqref{2.17} as function of time $t$.}
	\label{E_C0}
\end{figure}

\begin{problem}\label{pro2}
	(\textbf{Strong magnetic field with $q=2$.})
	In this problem, we illustrate the behaviour of the numerical schemes for solving \eqref{1.1} with $q=2$. Consider the following strong magnetic field
	\begin{equation*}
		\frac{1}{\varepsilon}B=\frac{1}{\varepsilon}
		\begin{pmatrix}
			1 - \sin(\varepsilon^2 x_{2})/2 \\ 
			1 + \cos(\varepsilon^2 x_{3})/2 \\
			1 + \cos(\varepsilon^2 x_{1})/2
		\end{pmatrix},
	\end{equation*}
	and the electric field \(E(x)=-\nabla U(x)\) with the potential \(U(x)=\dfrac{1}{\sqrt{x_{1}^{2}+x_{2}^{2}+x_{3}^{2}}}\). The initial values are \(x(0)=(1/6,1/8,1/4)^{\intercal}\) and \(v(0)=(1/5,1/3,1/2)^{\intercal}\).
\end{problem}

The relative errors \eqref{2.15}-\eqref{2.16} and the energy errors \eqref{2.17} of S2-new \eqref{2.5} and S2-VP \eqref{2.19} are shown in Figs.~\ref{S2_C2}-\ref{E_C2}. Fig.~\ref{S2_C2} shows that when $q=2$, the S2-new has a second-order error bound with respect to the step size $h$, and this error bound is independent of $\varepsilon$, which verifies the theoretical result in \eqref{2.7} of Theorem \ref{th2.1}.
And compared with the S2-VP (see Fig.~\ref{VP_C2}), our scheme demonstrates superior performance.
Fig.~\ref{E_C2} further illustrates the long-term energy near preservation of both schemes.

\begin{figure}[t!] 
	\centering
	\begin{subfigure}{0.328\textwidth}
		\centering
		\includegraphics[width=\linewidth]{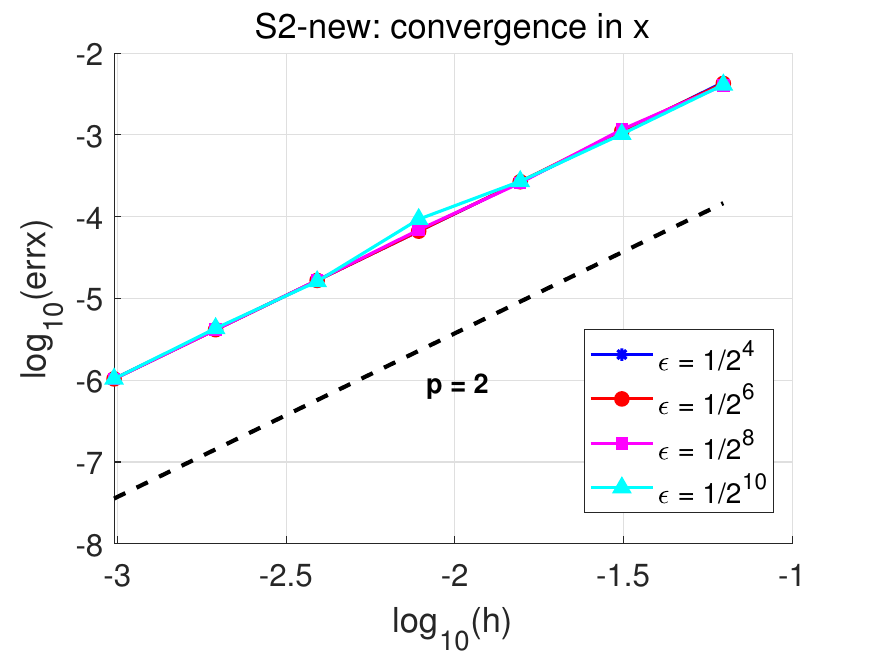}
	\end{subfigure}
	\begin{subfigure}{0.328\textwidth}
		\centering
		\includegraphics[width=\linewidth]{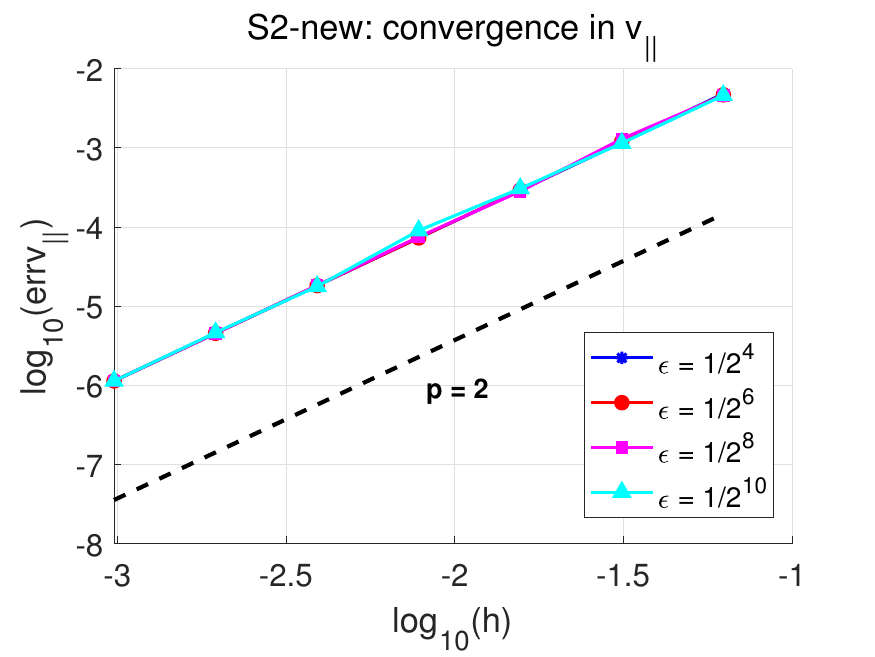}  
	\end{subfigure}
	\begin{subfigure}{0.328\textwidth}
		\centering
		\includegraphics[width=\linewidth]{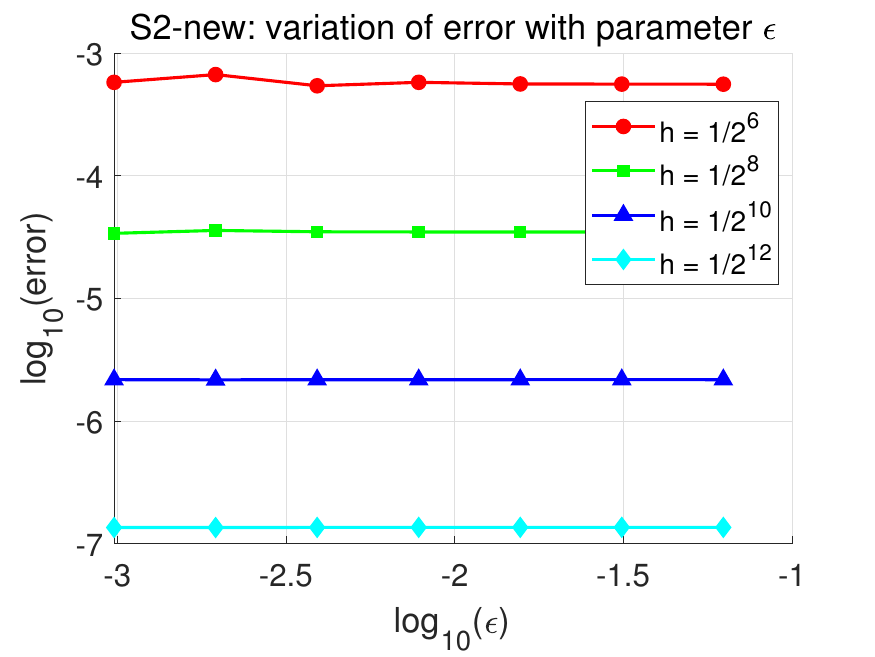}
	\end{subfigure}
	\caption{Problem \ref{pro2}. The error \eqref{2.15} (left and middle) of S2-new with $h = 1/2^{k}$ for $k=4,...,10$ under different $\varepsilon$ in \eqref{1.1} with $q=2$. The error \eqref{2.16} (right) of S2-new with $\varepsilon = 1/2^{k}$ for $k=4,...,10$ under different $h$ in \eqref{1.1} with $q=2$.}
	\label{S2_C2}
\end{figure}

\begin{figure}[t!] 
	\centering
	\begin{subfigure}{0.328\textwidth}
		\centering
		\includegraphics[width=\linewidth]{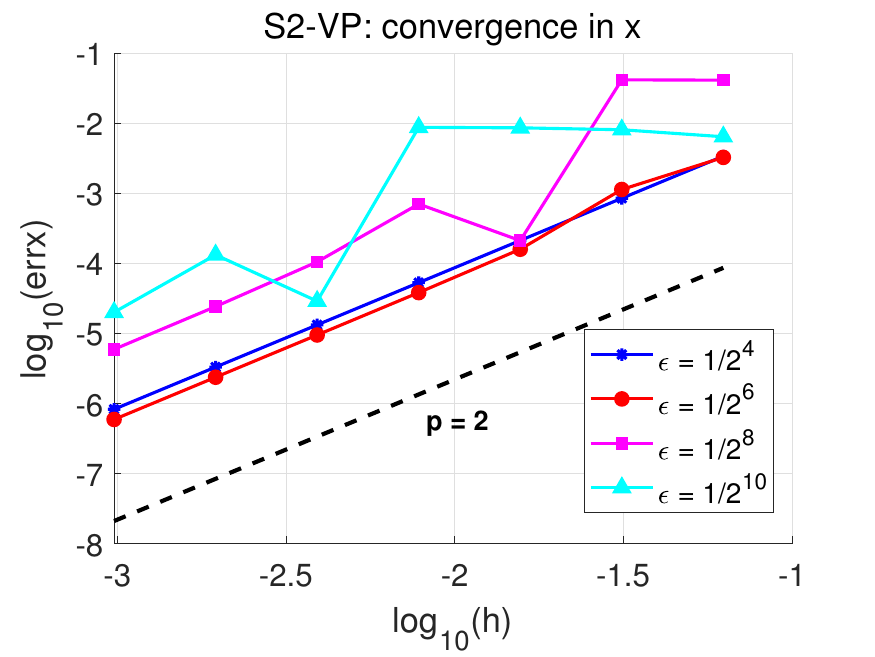}
	\end{subfigure}
	\begin{subfigure}{0.328\textwidth}
		\centering
		\includegraphics[width=\linewidth]{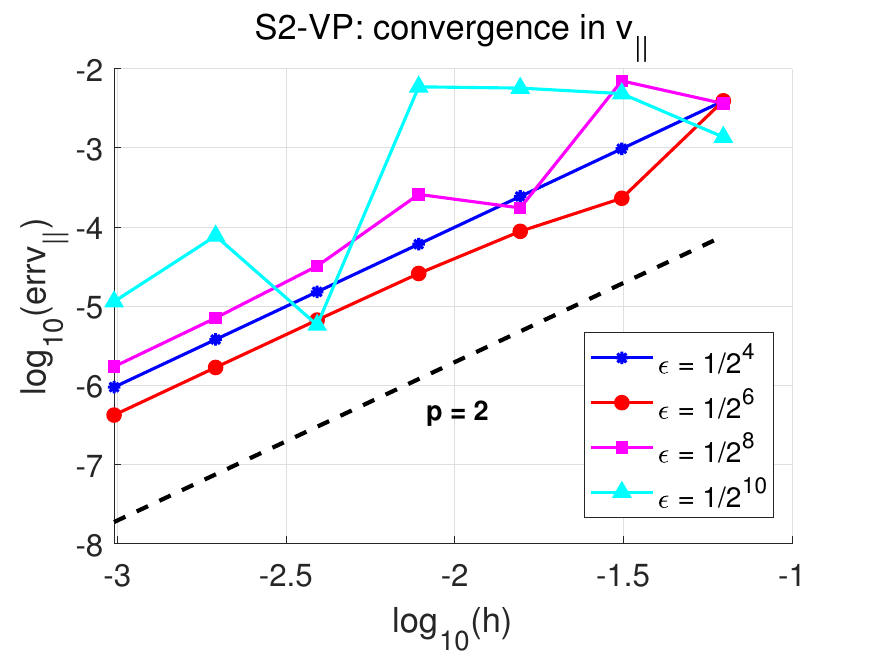}  
	\end{subfigure}
	\begin{subfigure}{0.328\textwidth}
		\centering
		\includegraphics[width=\linewidth]{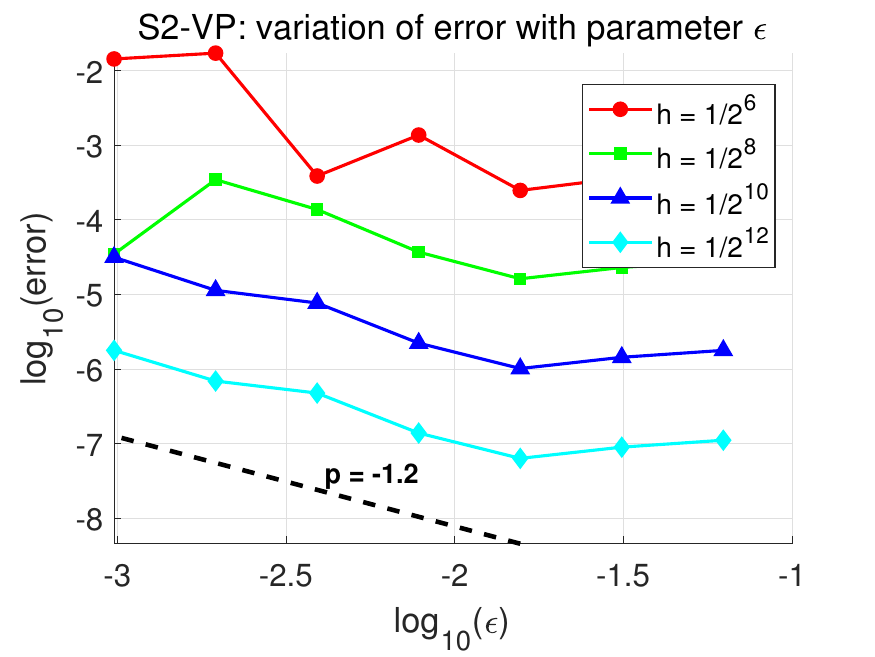}
	\end{subfigure}
	\caption{Problem \ref{pro2}. The error \eqref{2.15} (left and middle) of S2-VP with $h = 1/2^{k}$ for $k=4,...,10$ under different $\varepsilon$ in \eqref{1.1} with $q=2$. The error \eqref{2.16} (right) of S2-VP with $\varepsilon = 1/2^{k}$ for $k=4,...,10$ under different $h$ in \eqref{1.1} with $q=2$.}
	\label{VP_C2}
\end{figure}

\begin{figure}[t!] 
	\centering
	\begin{subfigure}{0.328\textwidth}
		\centering
		\includegraphics[width=\linewidth]{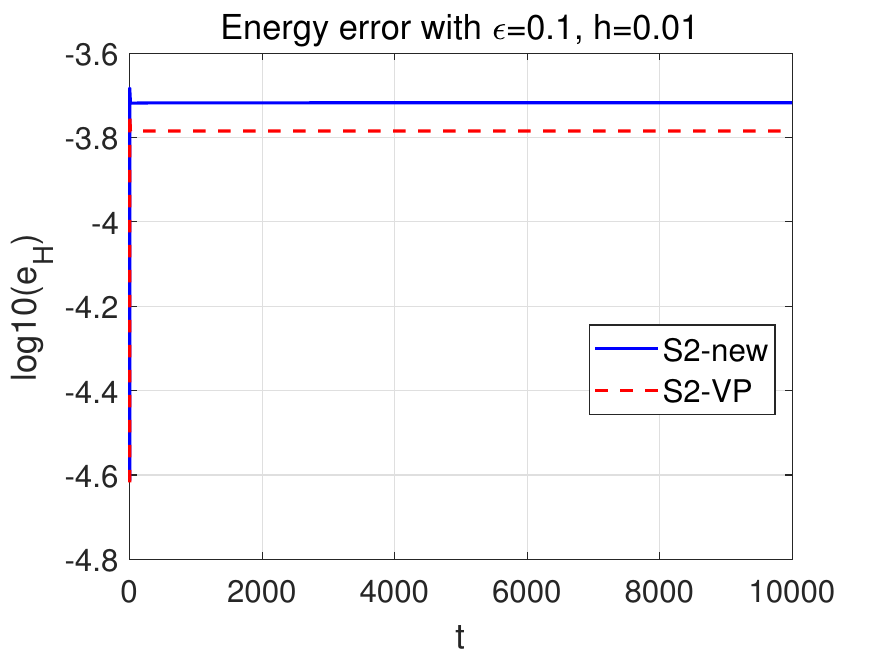}
	\end{subfigure}
	\begin{subfigure}{0.328\textwidth}
		\centering
		\includegraphics[width=\linewidth]{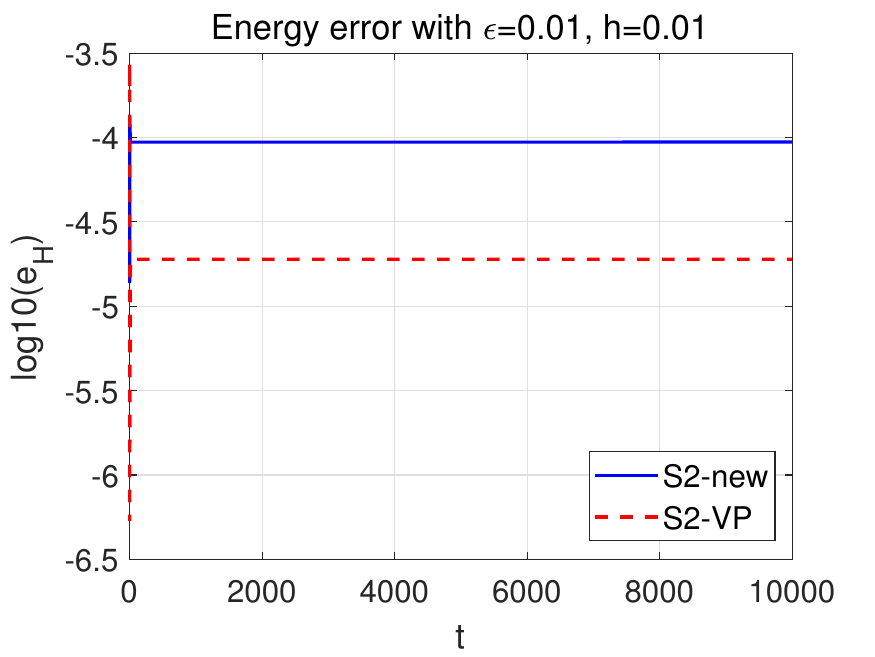}  
	\end{subfigure}
	\begin{subfigure}{0.328\textwidth}
		\centering
		\includegraphics[width=\linewidth]{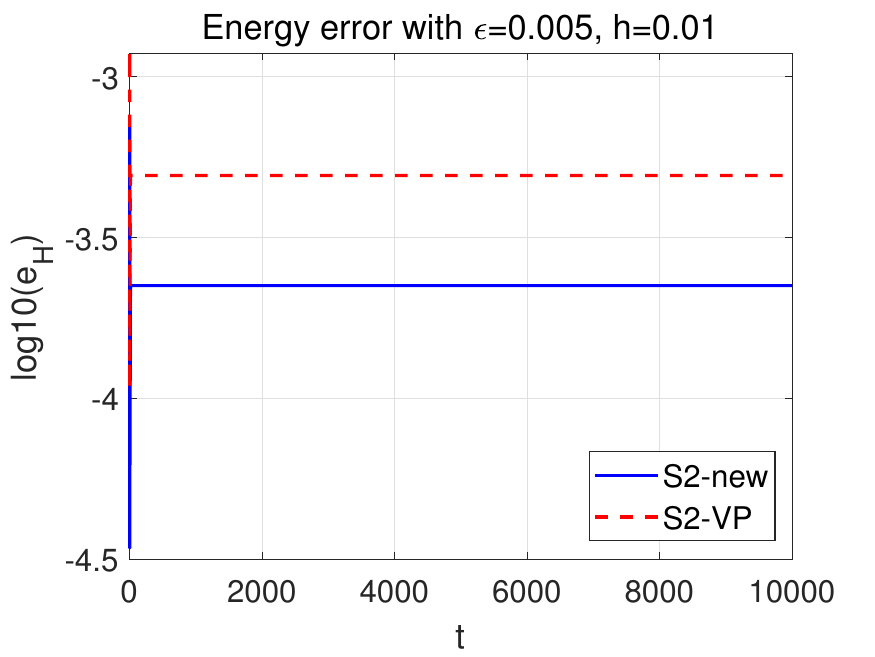}  
	\end{subfigure}
	\caption{Problem \ref{pro2}. Evolution of the energy error \eqref{2.17} as function of time $t$.}
	\label{E_C2}
\end{figure}

\begin{problem}\label{pro3}
	(\textbf{Strong magnetic field with $q=1.5$.})
	In \eqref{1.1} with $q=1.5$, we consider the following strong magnetic field
	\begin{equation*}
		\frac{1}{\varepsilon}B=\frac{1}{\varepsilon}
		\begin{pmatrix}
			1 - \sin(\varepsilon^{1.5} x_{2})/2 \\ 
			1 + \cos(\varepsilon^{1.5} x_{3})/2 \\
			1 + \cos(\varepsilon^{1.5} x_{1})/2
		\end{pmatrix},
	\end{equation*}
	and the electric field \(E(x)=-\nabla U(x)\) with the potential \(U(x)=x_{1}^{3}-x_{2}^{3}+\frac{1}{5}x_{1}^{4}+x_{2}^{4}+x_{3}^{4} \). The initial values are \(x(0)=(1/6,1/8,1/4)^{\intercal}\) and \(v(0)=(1/5,1/3,1/2)^{\intercal}\).
\end{problem}

The relative errors \eqref{2.15}-\eqref{2.16}are shown in Figs.~\ref{S2_C1.5}-\ref{VP_C1.5}. Fig.~\ref{S2_C1.5} shows that when $q=1.5$, the S2-new has a second-order error bound with respect to the step size $h$, and this error bound depends on   $\varepsilon^{-0.5}$, which performs better than the theoretical result \eqref{2.7} of Theorem \ref{th2.1}.
Fig.~\ref{VP_C1.5} shows that the S2-VP has a second-order error bound of $h$ depending on  $\varepsilon^{-1}$.
\begin{figure}[t!] 
	\centering
	\begin{subfigure}{0.328\textwidth}
		\centering
		\includegraphics[width=\linewidth]{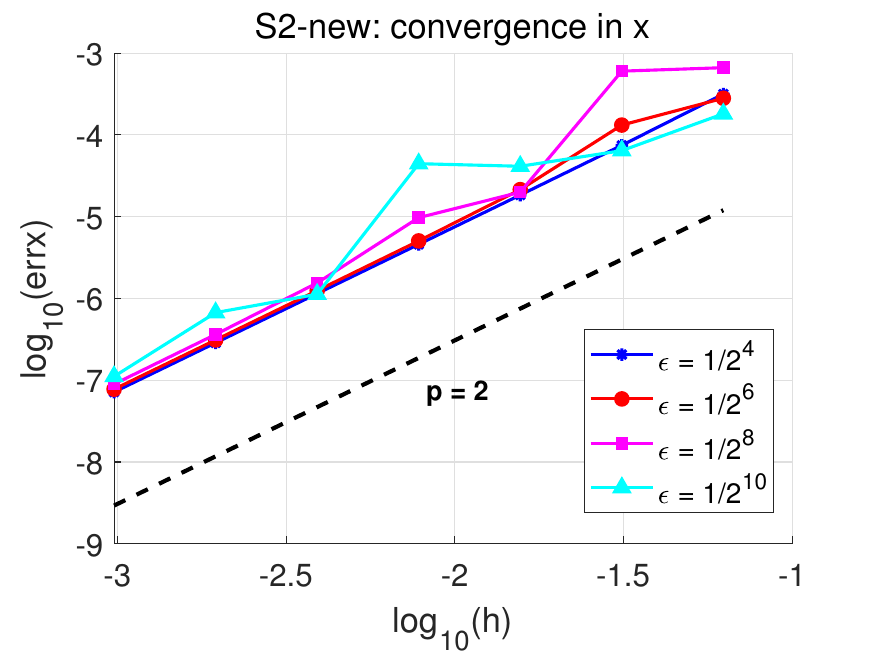}
	\end{subfigure}
	\begin{subfigure}{0.328\textwidth}
		\centering
		\includegraphics[width=\linewidth]{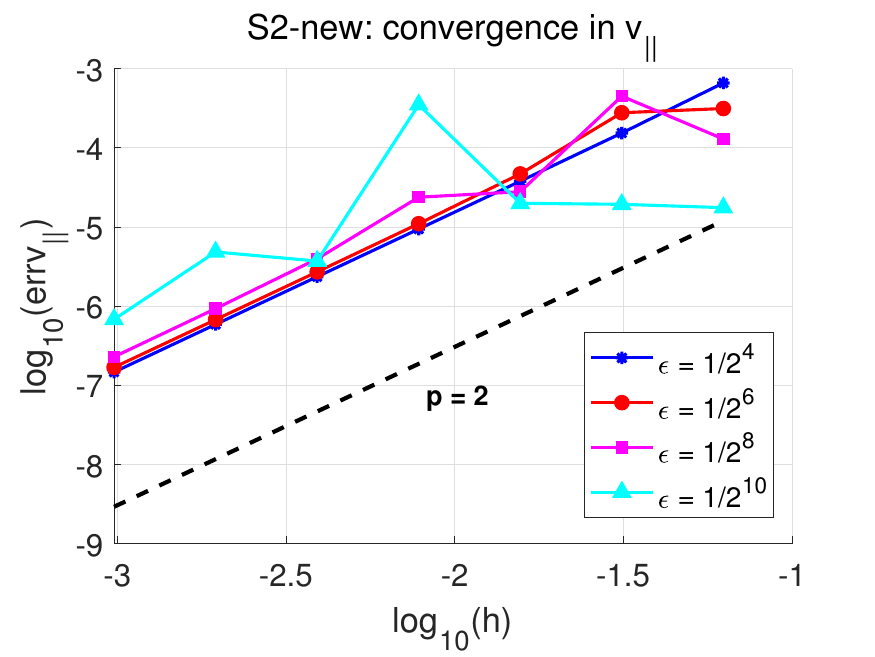}  
	\end{subfigure}
	\begin{subfigure}{0.328\textwidth}
		\centering
		\includegraphics[width=\linewidth]{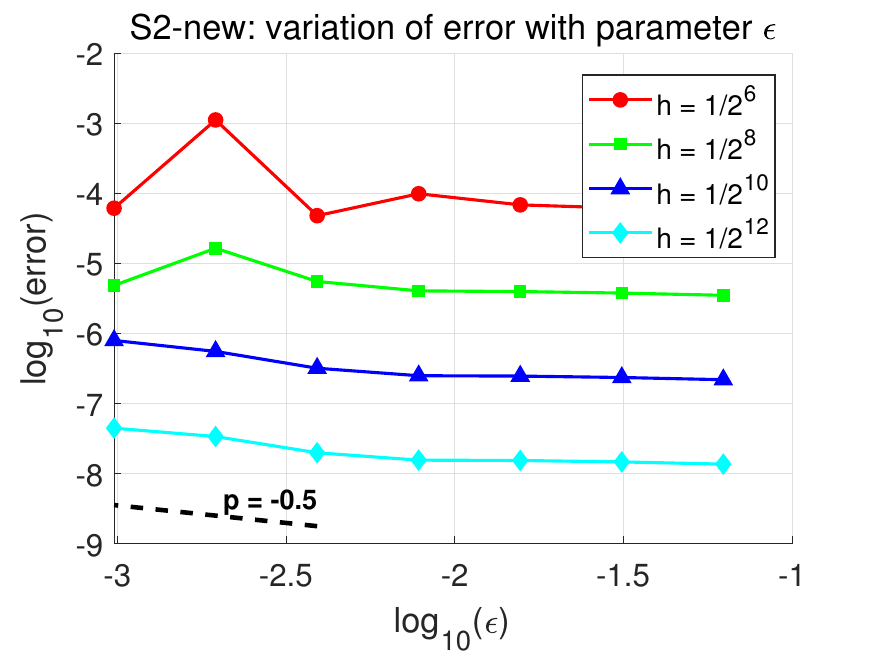}
	\end{subfigure}
	\caption{Problem \ref{pro3}. The error \eqref{2.15} (left and middle) of S2-new with $h = 1/2^{k}$ for $k=4,...,10$ under different $\varepsilon$ in \eqref{1.1} with $q=1.5$. The error \eqref{2.16} (right) of S2-new with $\varepsilon = 1/2^{k}$ for $k=4,...,10$ under different $h$ in \eqref{1.1} with $q=1.5$.}
	\label{S2_C1.5}
\end{figure}
\begin{figure}[t!] 
	\centering
	\begin{subfigure}{0.328\textwidth}
		\centering
		\includegraphics[width=\linewidth]{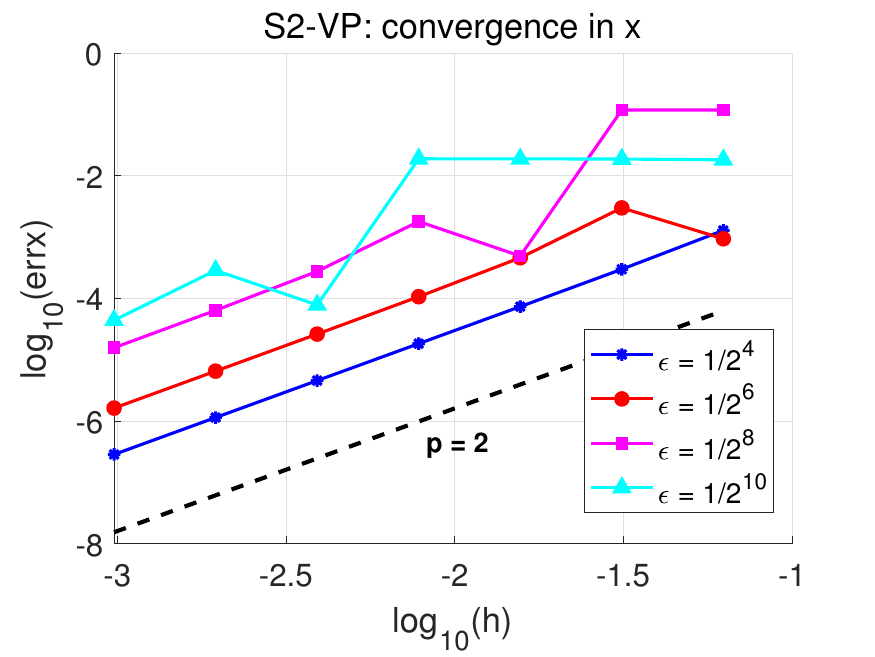}
	\end{subfigure}
	\begin{subfigure}{0.328\textwidth}
		\centering
		\includegraphics[width=\linewidth]{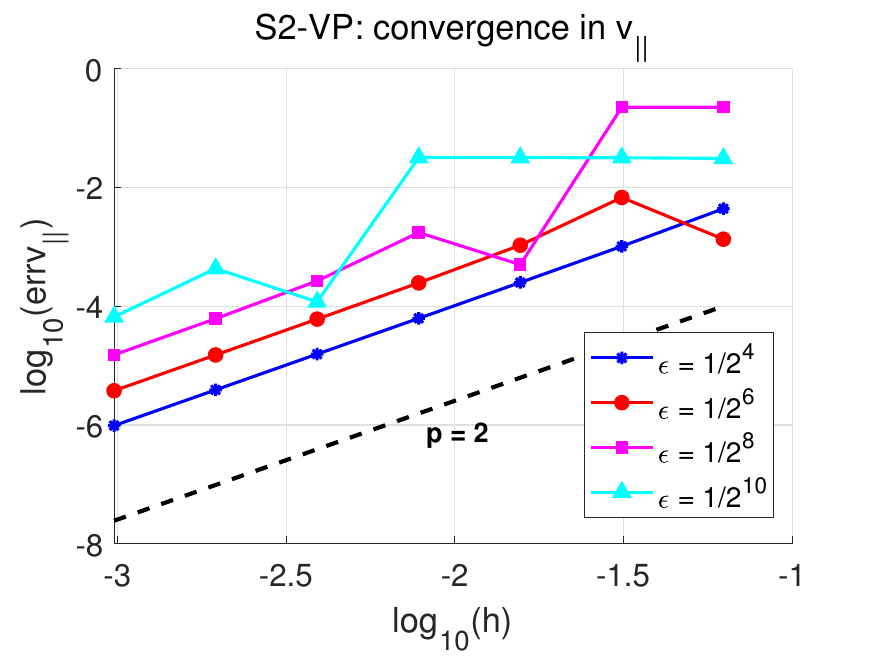}  
	\end{subfigure}
	\begin{subfigure}{0.328\textwidth}
		\centering
		\includegraphics[width=\linewidth]{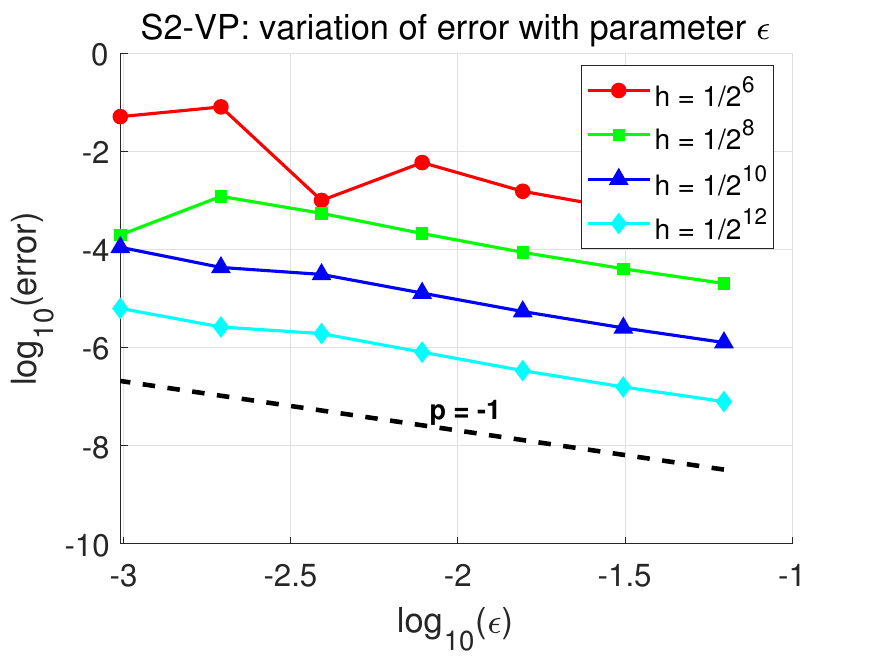}
	\end{subfigure}
	\caption{Problem \ref{pro3}. The error \eqref{2.15} (left and middle) of S2-VP with $h = 1/2^{k}$ for $k=4,...,10$ under different $\varepsilon$ in \eqref{1.1} with $q=1.5$. The error \eqref{2.16} (right) of S2-VP with $\varepsilon = 1/2^{k}$ for $k=4,...,10$ under different $h$ in \eqref{1.1} with $q=1.5$.}
	\label{VP_C1.5}
\end{figure}

\begin{problem}\label{pro4} (\textbf{Strong magnetic field with $q=1$.})
	Finally, for \eqref{1.1} with $q=1$, this problem adopts the following strong magnetic field
	\begin{equation*}
		\frac{1}{\varepsilon}B=\frac{1}{\varepsilon}
		\begin{pmatrix}
			1 - \sin(\varepsilon x_{2})/2 \\ 
			1 + \cos(\varepsilon x_{3})/2 \\
			1 + \cos(\varepsilon x_{1})/2
		\end{pmatrix},
	\end{equation*}
	and the electric field \(E(x)=-\nabla U(x)\) with the potential \(U(x)=-\sin(\frac{x_{1}}{2}) \sin(x_{2}) \sin(x_{3}) \). The initial values are \(x(0)=(1/6,1/8,1/4)^{\intercal}\) and \(v(0)=(1/5,1/3,1/2)^{\intercal}\).
\end{problem}
The relative errors \eqref{2.15}-\eqref{2.16} of S2-new \eqref{2.5} and S2-VP \eqref{2.19} are shown in Figs.~\ref{S2_C1}-\ref{VP_C1}. As shown in these results, when $q = 1$, the proposed scheme still demonstrates superior performance in terms of error results.
\begin{figure}[t!] 
	\centering
	\begin{subfigure}{0.328\textwidth}
		\centering
		\includegraphics[width=\linewidth]{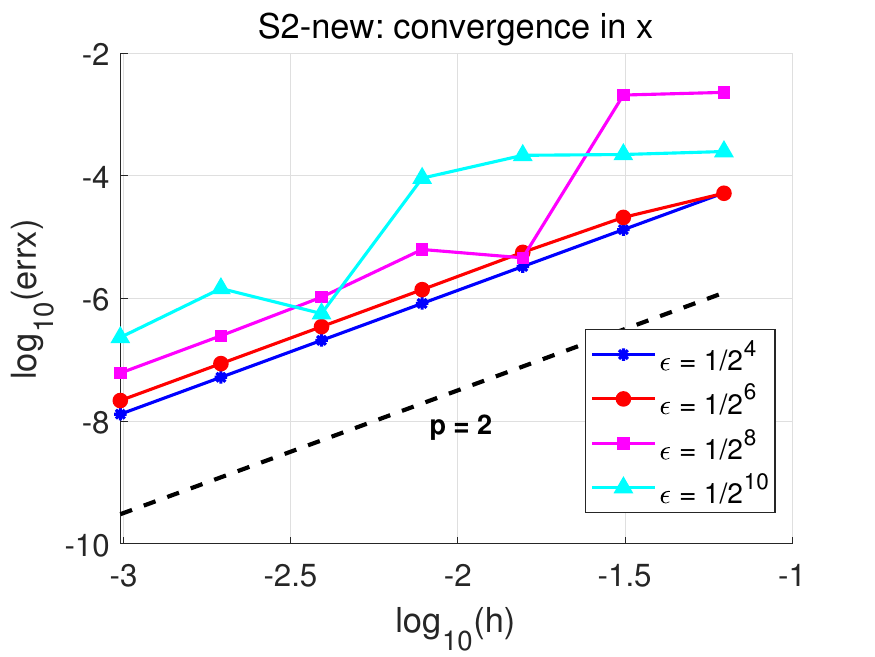}
	\end{subfigure}
	\begin{subfigure}{0.328\textwidth}
		\centering
		\includegraphics[width=\linewidth]{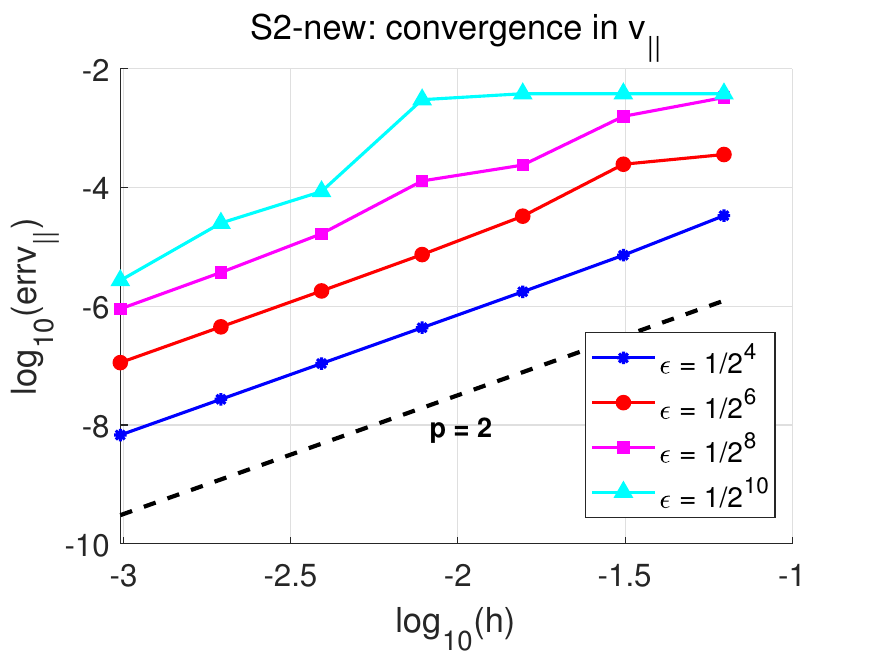}  
	\end{subfigure}
	\begin{subfigure}{0.328\textwidth}
		\centering
		\includegraphics[width=\linewidth]{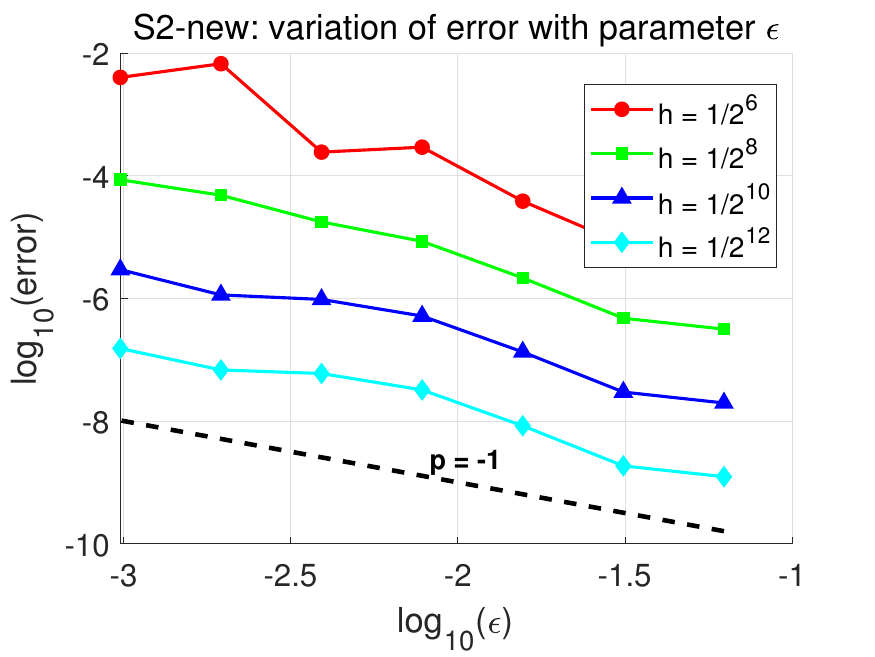}
	\end{subfigure}
	\caption{Problem \ref{pro4}. The error \eqref{2.15} (left and middle) of S2-new with $h = 1/2^{k}$ for $k=4,...,10$ under different $\varepsilon$ in \eqref{1.1} with $q=1$. The error \eqref{2.16} (right) of S2-new with $\varepsilon = 1/2^{k}$ for $k=4,...,10$ under different $h$ in \eqref{1.1} with $q=1$.}
	\label{S2_C1}
\end{figure}
\begin{figure}[t!] 
	\centering
	\begin{subfigure}{0.328\textwidth}
		\centering
		\includegraphics[width=\linewidth]{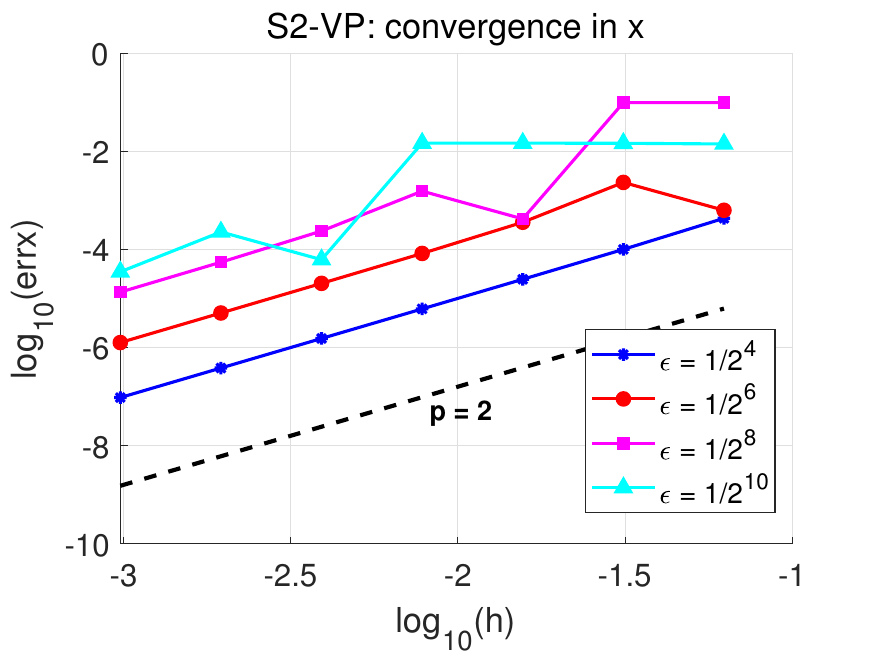}
	\end{subfigure}
	\begin{subfigure}{0.328\textwidth}
		\centering
		\includegraphics[width=\linewidth]{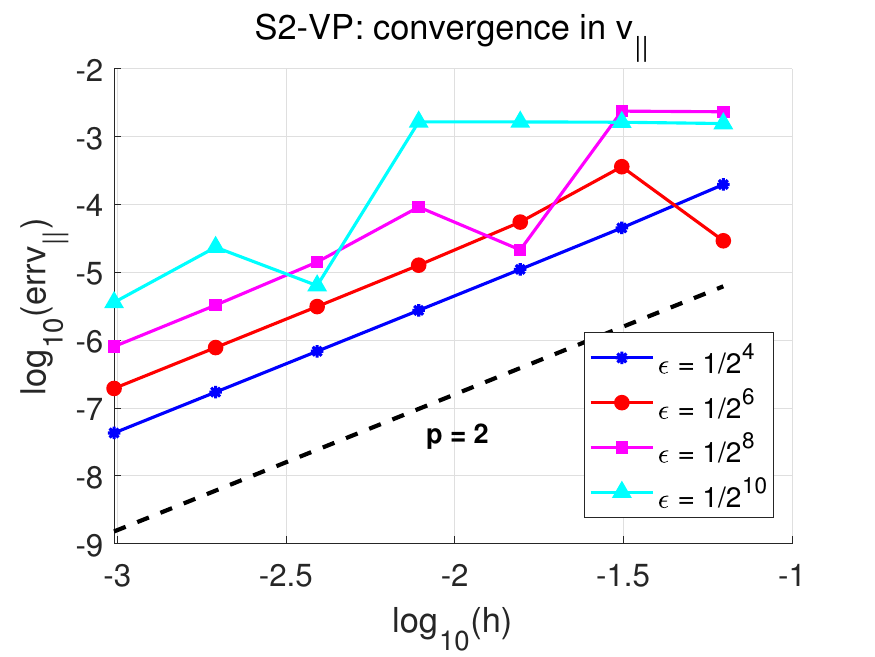}  
	\end{subfigure}
	\begin{subfigure}{0.328\textwidth}
		\centering
		\includegraphics[width=\linewidth]{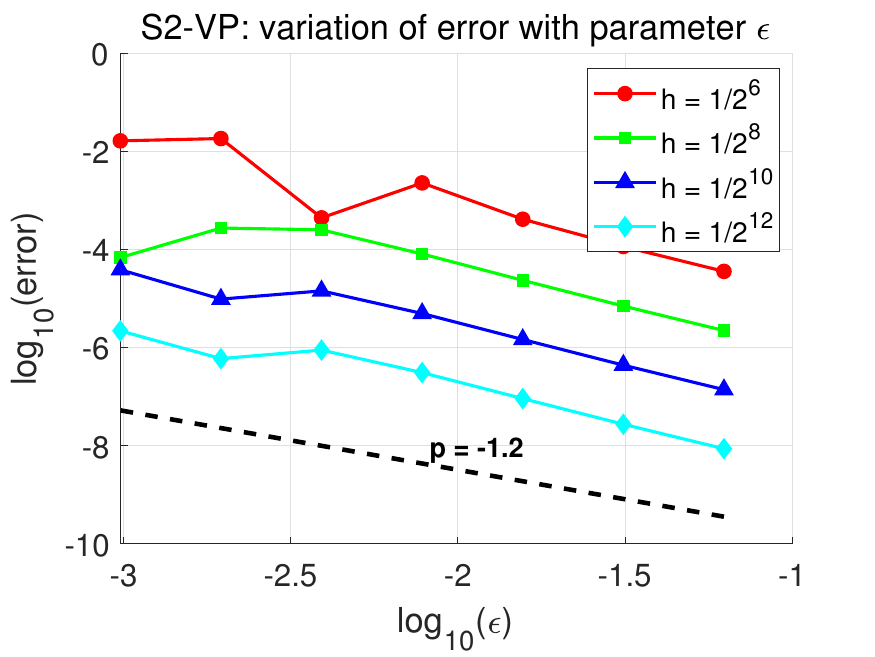}
	\end{subfigure}
	\caption{Problem \ref{pro4}. The error \eqref{2.15} (left and middle) of S2-VP with $h = 1/2^{k}$ for $k=4,...,10$ under different $\varepsilon$ in \eqref{1.1} with $q=1$. The error \eqref{2.16} (right) of S2-VP with $\varepsilon = 1/2^{k}$ for $k=4,...,10$ under different $h$ in \eqref{1.1} with $q=1$.}
	\label{VP_C1}
\end{figure}

\section{Error estimates (proof of Theorem \ref{th2.1})}\label{sec3}

This section provides a rigorous proof of Theorem \ref{th2.1}, with the proof framework outlined as follows:

\begin{itemize}
	\item Section \ref{subsec3.1} reformulates the original system as a long-time problem via a time rescaling technique.
	\item Section \ref{subsec3.2} introduces a truncated system for the long-time problem and establishes the associated approximation error.
	\item The local truncation error results for the S2-new are presented in Section \ref{subsec3.3}.
	\item Section \ref{subsec3.4} provides a preliminary error estimate of the numerical solution. 
	\item The improved error bounds for the numerical solution are obtained in Section \ref{subsec3.5}.
\end{itemize}

\subsection{Time rescaling}\label{subsec3.1}

We first focus our analysis on the CPD \eqref{1.1}
\begin{equation}
	\dot{x} = v, \quad \dot{v} = v \times \frac{B(\varepsilon^q x)}{\varepsilon} + E(x), \quad 1 \leq q \leq 2, \quad 0 < t \leq T,
	\label{3.1}
\end{equation}
where the magnetic field \( B(\varepsilon^q x) \) satisfies the condition \( |B(0)| > 0 \) independent of \(\varepsilon\).

To establish error bounds with optimal dependence on \(\varepsilon\) for the proposed method on the time interval \([0,T]\) (where \(T\) is independent of \(\varepsilon\)), we introduce the following time rescaling
\begin{equation*}
	\tau := t/\varepsilon, \quad z(\tau) := x(t), \quad w(\tau) := v(t), \quad 0 \leq \tau \leq \frac{T}{\varepsilon}.
\end{equation*}
This transformation reformulates \eqref{3.1} into a long-time problem
\begin{equation}
	\begin{cases}
		\dot{z}(\tau) = \varepsilon w(\tau),  \\
		\dot{w}(\tau) = w(\tau) \times B(\varepsilon^q z(\tau)) + \varepsilon E(z(\tau)), \quad 1 \leq q \leq 2, 
		\displaystyle \quad 0 < \tau \leq \frac{T}{\varepsilon}, \\
		z(0) = x_0, \quad w(0) = v_0.
	\end{cases}
	\label{3.3}
\end{equation}
This formulation over long-time scales facilitates the resolution of scale disparities between \(\varepsilon\) and the time step, while also revealing the averaging effect to be utilized in later analysis. Under the assumption that \(E(\cdot)\) and \(B(\cdot)\) belong to \(C^2(\mathbb{R}^3)\), it is evident that for \eqref{3.3} we obtain 
\begin{equation*}
	\|z\|_{L^{\infty}(0,T/\varepsilon)} + \|w\|_{L^{\infty}(0,T/\varepsilon)} \lesssim 1.
\end{equation*}
For the long-time problem \eqref{3.3}, define the time step size under the new scaling as \(\mathfrak{h} > 0\), and \(\tau_n = n\mathfrak{h}\) for \(n\in\mathbb{N}\). Given \(z^0 = x_0\) and \(w^0 = v_0\), and denoting the numerical solutions as \(z^n \approx z(\tau_n)\) and \(w^n \approx w(\tau_n)\), the splitting scheme S2-new for solving the long-time problem \eqref{3.3} is given by 
\begin{align}
	\begin{split}
		\displaystyle z^{n+1}=& z^{n}+\frac{\varepsilon\mathfrak{h}}{2}\varphi_{1}\biggl(\frac{\mathfrak{h}\widehat{B}_{0}}{2}\biggr)\Bigl(I+\mathrm{e}^{\mathfrak{h}(\widehat{B}_{\bar{z}_{n}}-\widehat{B}_{0})}\mathrm{e}^{\frac{\mathfrak{h}\widehat{B}_{0}}{2}}\Bigr)w^{n} \\
		\displaystyle &+\frac{\varepsilon^{2}\mathfrak{h}^{2}}{2}\varphi_{1}\biggl(\frac{\mathfrak{h}\widehat{B}_{0}}{2}\biggr)
		\varphi_{1}\biggl(\mathfrak{h}(\widehat{B}_{\bar{z}_{n}}-\widehat{B}_{0})\biggr)
		E(\bar{z}_{n}), \quad 0 \leq n < \frac{T}{\varepsilon\mathfrak{h}}, \\
		\displaystyle w^{n+1}=& \mathrm{e}^{\frac{\mathfrak{h}\widehat{B}_{0}}{2}}
		\mathrm{e}^{\mathfrak{h}(\widehat{B}_{\bar{z}_{n}}-\widehat{B}_{0})}
		\mathrm{e}^{\frac{\mathfrak{h}\widehat{B}_{0}}{2}}w^{n} +\varepsilon\mathfrak{h}\mathrm{e}^{\frac{\mathfrak{h}\widehat{B}_{0}}{2}}
		\varphi_{1}\biggl(\mathfrak{h}(\widehat{B}_{\bar{z}_{n}}-\widehat{B}_{0})\biggr)
		E(\bar{z}_{n}),
	\end{split}
	\label{3.5}
\end{align}
with the updated notations 
\begin{equation}
	\bar{z}_{n}=z^{n}+\frac{1}{2}\varepsilon\mathfrak{h}\varphi_{1}\biggl(\frac{\mathfrak{h}\widehat{B}_{0}}{2}\biggr)w^{n},
	\quad \widehat{B}_{0}=\widehat{B}(\varepsilon^q x_{0}),
	\quad \widehat{B}_{\bar{z}_{n}}=\widehat{B}(\varepsilon^q \bar{z}_{n}),
	\quad 1 \leq q \leq 2.
	\label{C1}
\end{equation}

\subsection{Linearization of the problem.}\label{subsec3.2}

To prove \eqref{2.7} in Theorem \ref{th2.1}, we initiate the proof by linearizing the problem. 
We consider a truncated system of \eqref{3.3} at time \(\tau = \tau_n + s\) as
\begin{equation}
	\begin{cases}
		\dot{\widetilde{z}^n}(s) = \varepsilon \widetilde{w}^n(s), \quad 0 < s \leq \mathfrak{h}, \\
		\displaystyle \dot{\widetilde{w}^n}(s) = \widetilde{w}^n(s) \times B(\varepsilon^q z(\tau_n + \frac{\mathfrak{h}}{2})) + \varepsilon E(\widetilde{z}^n(s)), \quad 1 \leq q \leq 2, 
		\displaystyle \quad 0 \leq n < \frac{T}{\varepsilon \mathfrak{h}}, \\
		\widetilde{z}^n(0) = z(\tau_n),  \quad  \widetilde{w}^n(0) = w(\tau_n).
	\end{cases}
	\label{3.7}
\end{equation}
Moreover, there exists a uniform constant \(C > 0\) that depends on \(\|z\|_{L^{\infty}(0,T/\varepsilon)}\), \(\|w\|_{L^{\infty}(0,T/\varepsilon)}\) and the norms of \(B\) and \(E\), such that
\begin{equation*}
	\|\widetilde{z}^n\|_{L^{\infty}(0,\mathfrak{h})} + \|\widetilde{w}^n\|_{L^{\infty}(0,\mathfrak{h})} \leq C, 
	\displaystyle \quad 
	0 \leq n < \frac{T}{\varepsilon \mathfrak{h}}.
\end{equation*}

The errors between the long-time problem \eqref{3.3} and the truncated system \eqref{3.7} are defined as follows
\begin{equation}
	\zeta^n_{z}(s) := z(\tau_n + s) - \widetilde{z}^n(s), \quad 
	\zeta^n_{w}(s) := w(\tau_n + s) - \widetilde{w}^n(s), \quad 
	0 \leq n < \frac{T}{\varepsilon \mathfrak{h}}.
	\label{A}
\end{equation}
And taking the difference between \eqref{3.3} and \eqref{3.7}, we get 
\begin{equation}
	\begin{cases}
		\dot{\zeta}^n_{z}(s) = \varepsilon \zeta^n_{w}(s), \quad 0 < s \leq \mathfrak{h},  
		\displaystyle \quad 0 \leq n < \frac{T}{\varepsilon \mathfrak{h}}, \\
		\displaystyle \dot{\zeta}^n_{w}(s) = \zeta^n_{w}(s) \times B(\varepsilon^q z(\tau_n + \frac{\mathfrak{h}}{2})) + \varepsilon E(z(\tau_n + s)) - \varepsilon E(\widetilde{z}^n(s)) + \xi_{0}^{n}(s), \\
		\displaystyle \zeta^n_{z}(0) = \zeta^n_{w}(0) = 0,
		\label{3.8}
	\end{cases}
\end{equation}
where
\begin{equation*}
	\xi_{0}^{n}(s) = w(\tau_n + s) \times \left[ B(\varepsilon^q z(\tau_n + s)) - B(\varepsilon^q z(\tau_n + \frac{\mathfrak{h}}{2})) \right],
	\quad 1 \leq q \leq 2.
\end{equation*}

Using the Duhamel's formula for \eqref{3.8}, we obtain
\begin{subequations}
	\begin{align}
		\zeta_{z}^{n}(\mathfrak{h}) &= \varepsilon \int_{0}^{\mathfrak{h}} \zeta_{w}^{n}(s) \, ds, \quad 0 \leq n < \frac{T}{\varepsilon \mathfrak{h}}, \label{3.9a} \\
		\zeta_{w}^{n}(\mathfrak{h}) &= \int_{0}^{\mathfrak{h}} \mathrm{e}^{(\mathfrak{h}-s)\widehat{B}(\varepsilon^q z(\tau_{n}+\frac{\mathfrak{h}}{2}))} 
		\bigl[\varepsilon E(z(\tau_{n}+s)) - \varepsilon E(\widetilde{z}^{n}(s)) + \xi_{0}^{n}(s)\bigr] \, ds \notag \\
		& = \int_{0}^{\mathfrak{h}} \mathrm{e}^{(\mathfrak{h}-s)\widehat{B}(\varepsilon^q z(\tau_{n}+\frac{\mathfrak{h}}{2}))} 
		\left[ \varepsilon \int_{0}^{1} \nabla E\bigl(z(\tau_{n}+s) + (\rho-1)\zeta_{z}^{n}(s)\bigr)\zeta_{z}^{n}(s) \, d\rho + \xi_{0}^{n}(s) \right] ds. 
		\label{3.9b}
	\end{align}
\end{subequations}
Substituting \eqref{3.9b} into \eqref{3.9a}, \(\zeta_{z}^{n}(\mathfrak{h})\) can be rewritten as
\begin{align*}
	\zeta_{z}^{n}(\mathfrak{h}) = & \varepsilon^{2} \int_{0}^{\mathfrak{h}} \int_{0}^{s} \mathrm{e}^{(s-\sigma)\widehat{B}(\varepsilon^q z(\tau_{n}+\frac{\mathfrak{h}}{2}))} 
	\int_{0}^{1} \nabla E\bigl(z(\tau_{n}+\sigma) + (\rho-1)\zeta_{z}^{n}(\sigma)\bigr)\zeta_{z}^{n}(\sigma) \, d\rho \, d\sigma \, ds  \\
	& + \varepsilon \int_{0}^{\mathfrak{h}} \int_{0}^{s} \mathrm{e}^{(s-\sigma)\widehat{B}(\varepsilon^q z(\tau_{n}+\frac{\mathfrak{h}}{2}))} \xi_{0}^{n}(\sigma) \, d\sigma \, ds.	
\end{align*}
Employing the fundamental principle of calculus and introducing the abbreviation \(\widehat{B}_n = \widehat{B}(\varepsilon^q z(\tau_n + \frac{\mathfrak{h}}{2}))\), it is straightforward to derive that 
\begin{align*}
	&\left| \int_0^{\mathfrak{h}} \mathrm{e}^{(\mathfrak{h}-s)\widehat{B}_n} \xi^n_0(s) \, ds \right| \notag \\
	&\,\,= \left| \int_0^{\mathfrak{h}} \mathrm{e}^{(\mathfrak{h}-s)\widehat{B}_n} w(\tau_n + s) \times 
	\left[ B(\varepsilon^q z(\tau_n + s)) - B(\varepsilon^q z(\tau_n + \frac{\mathfrak{h}}{2})) \right] ds \right| \notag \\
	&\,\,= \left| \int_0^{\mathfrak{h}} \mathrm{e}^{(\mathfrak{h}-s)\widehat{B}_n} w(\tau_n + s) \times 
	\left[ \varepsilon^{q+1} \int_{\mathfrak{h}/2}^s \nabla B(\varepsilon^q z(\tau_n + \rho)) w(\tau_n + \rho) \, d\rho \right] ds \right| \notag \\
	&\,\, = \varepsilon^{q+1} \left| \int_0^{\mathfrak{h}} \mathrm{e}^{(\mathfrak{h}-s)\widehat{B}_n} F_n(s) \, ds \right|,
\end{align*}
where
\begin{equation*}
	F_n(s) = w(\tau_n + s) \times \int_{\mathfrak{h}/2}^s \nabla B(\varepsilon^q z(\tau_n + \rho)) w(\tau_n + \rho) \, d\rho.
\end{equation*}
Utilizing the midpoint integration rule and observing that \( F_n(\mathfrak{h}/2) = 0 \), we have
\begin{align*}
	& \left| \int_0^{\mathfrak{h}} \mathrm{e}^{(\mathfrak{h}-s)\widehat{B}(\varepsilon^q z(\tau_n + \frac{\mathfrak{h}}{2}))} \xi^n_0(s) \, ds \right| 
	= \varepsilon^{q+1} \left| \int_0^{\mathfrak{h}} \mathrm{e}^{(\mathfrak{h}-s)\widehat{B}_n} F_n(s) \, ds \right| \notag \\
	& \,\,= \varepsilon^{q+1} \left|  (\mathfrak{h}-0) \mathrm{e}^{(\mathfrak{h}-\frac{\mathfrak{h}}{2})\widehat{B}_n} F_n(\frac{\mathfrak{h}}{2}) + \frac{\mathfrak{h}^3}{24} \frac{d^2}{ds^2}(\mathrm{e}^{(\mathfrak{h}-s)\widehat{B}_n} F_n(s))|_{s=\xi} \right| \notag \\
	& \,\, \lesssim \varepsilon^{q+1} \mathfrak{h}^3,
\end{align*}
with $\xi \in (0,\mathfrak{h})$. Thus, it is derived
\begin{equation}
	|\zeta^n_w(\mathfrak{h})| \lesssim \varepsilon^{q+1} \mathfrak{h}^3,
	\quad 1 \leq q \leq 2, 
	\quad 0 \leq n < \frac{T}{\varepsilon \mathfrak{h}}.  
	\label{3.10}
\end{equation}
Substituting the above estimate into \eqref{3.9a} yields
\begin{equation}
	|\zeta^n_z(\mathfrak{h})| \lesssim \varepsilon^{q+2} \mathfrak{h}^4,
	\quad 1 \leq q \leq 2, 
	\quad 0 \leq n < \frac{T}{\varepsilon \mathfrak{h}}. 
	\label{3.11}
\end{equation}

In order to estimate the errors of S2-new \eqref{3.5}   
\begin{equation*}
	e_{z}^{n+1} := z(\tau_{n+1}) - z^{n+1}, \quad 
	e_{w}^{n+1} := w(\tau_{n+1}) - w^{n+1}, \quad 
	0 \leq n < \frac{T}{\varepsilon \mathfrak{h}},
\end{equation*} 
we reformulate the errors by using \eqref{A} 
\begin{align}
	e_{z}^{n+1} & = \zeta_{z}^{n}(\mathfrak{h}) + \widetilde{z}^{n}(\mathfrak{h}) - z^{n+1} \notag \\
	& = \zeta_{z}^{n}(\mathfrak{h}) + \widetilde{e}_{z}^{n}, 
	\quad 0 \leq n < \frac{T}{\varepsilon \mathfrak{h}},  \label{B1}\\
	e_{w}^{n+1} & = \zeta_{w}^{n}(\mathfrak{h}) + \widetilde{w}^{n}(\mathfrak{h}) - w^{n+1} \notag \\
	& = \zeta_{w}^{n}(\mathfrak{h}) + \widetilde{e}_{w}^{n}, 
	\quad 0 \leq n < \frac{T}{\varepsilon \mathfrak{h}}, \label{B2}
\end{align}  
and then proceed to estimate the terms defined in \eqref{B1}-\eqref{B2}
\begin{equation*}
	\widetilde{e}_{z}^{n} := \widetilde{z}^{n}(\mathfrak{h}) - z^{n+1}, \quad 
	\widetilde{e}_{w}^{n} := \widetilde{w}^{n}(\mathfrak{h}) - w^{n+1}, \quad 
	0 \leq n < \frac{T}{\varepsilon \mathfrak{h}}.
\end{equation*}

\subsection{Local error estimates.}\label{subsec3.3}
To derive the error bounds for the numerical solution, we present below the definition and proof of the local truncation errors for S2-new \eqref{3.5}.
\begin{lemma}\label{lemmalocal}
	Assume that \( E(\cdot), B(\cdot) \in C^2(\mathbb{R}^3) \) and the numerical solution at step $n-1$ is assumed to be exact, let \(\xi_{z}^{n}\) and \(\xi_{w}^{n}\) are the local truncation errors 	obtained from the S2-new for solving the truncated system \eqref{3.7} at step $n$. Then there exists a positive constant \( \mathfrak{h}_0 \) independent of \(\varepsilon\), such that when the time step \( 0 < \mathfrak{h} < \mathfrak{h}_0 \), we can obtain
	\begin{equation}
		|\xi_{z}^{n}| \lesssim \varepsilon^{3} \mathfrak{h}^{3}, \quad 
		|\xi_{w}^{n}| \lesssim \varepsilon^{2} \mathfrak{h}^{3}, 
		\quad 0 \leq n < \frac{T}{\varepsilon \mathfrak{h}}.
		\label{local}
	\end{equation}
\end{lemma}

\begin{proof}
	It can be observed that the long-time problem \eqref{3.3} and its truncated system \eqref{3.7} share a similar structure.
	According to the S2-new \eqref{3.5} and the definition of local truncation errors \(\xi_{z}^{n}\) and \(\xi_{w}^{n}\), we have
	\begin{subequations}\label{3.13}
		\begin{align}
			\widetilde{z}^{n}(\mathfrak{h}) = & z(\tau_{n})+\frac{\varepsilon\mathfrak{h}}{2}\varphi_{1}\biggl(\frac{\mathfrak{h}\widehat{B}_{0}}{2}\biggr)\Bigl(I+\mathrm{e}^{\mathfrak{h}(\widehat{B}_{\bar{z}(\tau_{n})}-\widehat{B}_{0})}\mathrm{e}^{\frac{\mathfrak{h}\widehat{B}_{0}}{2}}\Bigr)w(\tau_{n})   \notag  \\
			&+\frac{\varepsilon^{2}\mathfrak{h}^{2}}{2}\varphi_{1}\biggl(\frac{\mathfrak{h}\widehat{B}_{0}}{2}\biggr)
			\varphi_{1}\biggl(\mathfrak{h}(\widehat{B}_{\bar{z}(\tau_{n})}-\widehat{B}_{0})\biggr)
			E(\bar{z}(\tau_{n})) + \xi_{z}^{n}, 
			\quad 0 \leq n < \frac{T}{\varepsilon \mathfrak{h}}, \label{3.13a}\\
			\widetilde{w}^{n}(\mathfrak{h}) = & \mathrm{e}^{\frac{\mathfrak{h}\widehat{B}_{0}}{2}}
			\mathrm{e}^{\mathfrak{h}(\widehat{B}_{\bar{z}(\tau_{n})}-\widehat{B}_{0})}
			\mathrm{e}^{\frac{\mathfrak{h}\widehat{B}_{0}}{2}}w(\tau_{n})  \notag \\
			&+\varepsilon\mathfrak{h}\mathrm{e}^{\frac{\mathfrak{h}\widehat{B}_{0}}{2}}
			\varphi_{1}\biggl(\mathfrak{h}(\widehat{B}_{\bar{z}(\tau_{n})}-\widehat{B}_{0})\biggr) E(\bar{z}(\tau_{n})) + \xi_{w}^{n}, \label{3.13b}
		\end{align}
	\end{subequations}
	where 
	\begin{equation}
		\bar{z}(\tau_{n})=z(\tau_{n})+\frac{1}{2}\varepsilon\mathfrak{h}\varphi_{1}\biggl(\frac{\mathfrak{h}\widehat{B}_{0}}{2}\biggr)w(\tau_{n}), \quad \widehat{B}_{0}=\widehat{B}(\varepsilon^q x_{0}),
		\quad \widehat{B}_{\bar{z}(\tau_{n})}=\widehat{B}(\varepsilon^q \bar{z}(\tau_{n})),
		\quad 1 \leq q \leq 2. 
		\label{C2}
	\end{equation}
	By employing Duhamel’s formula for the truncated system \eqref{3.7}, it can be derived that
	\begin{subequations}
		\begin{align}
			\widetilde{z}^{n}(\mathfrak{h}) &= z(\tau_{n}) + \varepsilon \int_{0}^{\mathfrak{h}} \widetilde{w}^{n}(s) \, ds, 
			\quad 0 \leq n < \frac{T}{\varepsilon \mathfrak{h}}, \label{3.14a}\\
			\widetilde{w}^{n}(\mathfrak{h}) &= \mathrm{e}^{\mathfrak{h} \widehat{B}(\varepsilon^q z(\tau_{n}+\frac{\mathfrak{h}}{2}))} w(\tau_{n}) 
			+ \varepsilon \int_{0}^{\mathfrak{h}} \mathrm{e}^{(\mathfrak{h}-s)\widehat{B}(\varepsilon^q z(\tau_{n}+\frac{\mathfrak{h}}{2}))} E(\widetilde{z}^{n}(s)) \, ds. \label{3.14b}
		\end{align}
	\end{subequations}
	Substituting \eqref{3.14b} into \eqref{3.14a}, \(\widetilde{z}^{n}(\mathfrak{h})\) can be rewritten as
	\begin{equation}
		\begin{aligned}
			\widetilde{z}^{n}(\mathfrak{h}) = & z(\tau_{n}) + \varepsilon \int_{0}^{\mathfrak{h}} \mathrm{e}^{s \widehat{B}(\varepsilon^q z(\tau_{n}+\frac{\mathfrak{h}}{2}))} w(\tau_{n}) \, ds \\
			& + \varepsilon^{2} \int_{0}^{\mathfrak{h}} \int_{0}^{s} \mathrm{e}^{(s-\sigma)\widehat{B}(\varepsilon^q z(\tau_{n}+\frac{\mathfrak{h}}{2}))} E(\widetilde{z}^{n}(\sigma)) \, d\sigma \, ds.
		\end{aligned}
		\label{3.15}
	\end{equation}
	
	$\bullet$   \textbf{Estimate of $\xi_{w}^{n}$.}    Subtracting \eqref{3.14b} from \eqref{3.13b}, we can find that
	\begin{equation*}
		\xi_{w}^{n} = \xi_{w,1}^{n} + \xi_{w,2}^{n}, 
		\quad 0 \leq n < \frac{T}{\varepsilon \mathfrak{h}},
	\end{equation*}
	with
	\begin{subequations}
		\begin{align}
			\xi_{w,1}^{n} & = \mathrm{e}^{\mathfrak{h} \widehat{B}(\varepsilon^q z(\tau_{n}+\frac{\mathfrak{h}}{2}))} w(\tau_{n})
			- \mathrm{e}^{\frac{\mathfrak{h}\widehat{B}_{0}}{2}}
			\mathrm{e}^{\mathfrak{h}(\widehat{B}_{\bar{z}(\tau_{n})}-\widehat{B}_{0})}
			\mathrm{e}^{\frac{\mathfrak{h}\widehat{B}_{0}}{2}}w(\tau_{n}),  \label{3.16a} \\
			\xi_{w,2}^{n} & = \varepsilon \int_{0}^{\mathfrak{h}} \mathrm{e}^{(\mathfrak{h}-s)\widehat{B}(\varepsilon^q z(\tau_{n}+\frac{\mathfrak{h}}{2}))} E(\widetilde{z}^{n}(s)) \, ds
			- \varepsilon\mathfrak{h}\mathrm{e}^{\frac{\mathfrak{h}\widehat{B}_{0}}{2}}
			\varphi_{1}\biggl(\mathfrak{h}(\widehat{B}_{\bar{z}(\tau_{n})}-\widehat{B}_{0})\biggr) E(\bar{z}(\tau_{n})).	 \label{3.16b}
		\end{align}
	\end{subequations}
	Applying Taylor expansion with integral remainder, i.e.,  
	\begin{equation*}
		\mathrm{e}^{\mathfrak{h} \widehat{B}} = \sum_{i=0}^{k-1} \frac{\mathfrak{h}^i}{i!} \widehat{B}^i 
		+ \frac{1}{(k-1)!} \int_{0}^{\mathfrak{h}} e^{t\widehat{B}} \widehat{B}^k (\mathfrak{h} - t)^{k-1} \, dt,
		\quad k \ge 1,
	\end{equation*}
	it can be derived from \eqref{3.16a} that
	\begin{equation*}
		\begin{aligned}
			\xi_{w,1}^{n} = & \biggl( I + \mathfrak{h}\widehat{B}_{n} + \frac{\mathfrak{h}^2}{2!} \widehat{B}^2_{n} + \cdots \biggr) w(\tau_{n})
			- \biggl( I + \frac{\mathfrak{h}}{2}\widehat{B}_{0} + \frac{\mathfrak{h}^2}{4 \cdot 2!}\widehat{B}^{2}_{0} + \cdots \biggr)	\\
			& \biggl( I + \mathfrak{h}(\widehat{B}_{\bar{z}(\tau_{n})}-\widehat{B}_{0}) + \frac{\mathfrak{h}^2}{2!}(\widehat{B}_{\bar{z}(\tau_{n})}-\widehat{B}_{0})^2 + \cdots \biggr)	\biggl( I + \frac{\mathfrak{h}}{2}\widehat{B}_{0} + \frac{\mathfrak{h}^2}{4 \cdot 2!}\widehat{B}^{2}_{0} + \cdots \biggr) w(\tau_{n}) \\
			= & \, \mathfrak{h} (\widehat{B}_{n}-\widehat{B}_{\bar{z}(\tau_{n})}) w(\tau_{n})
			+ \frac{\mathfrak{h}^2}{2} (\widehat{B}^2_{n}-\widehat{B}^{2}_{\bar{z}(\tau_{n})})  w(\tau_{n}) +\cdots .
		\end{aligned}
	\end{equation*}
	Owing to the conditions \(B(\cdot) \in C^2(\mathbb{R}^3)\) and \(1 \leq q \leq 2\), the use of Taylor expansion yields
	\begin{align} 
		& \left| \widehat{B}_{n} - \widehat{B}_{\bar{z}(\tau_{n})} \right| 
		= \left|\nabla\widehat{B} \cdot \biggl(\varepsilon^q z(\tau_{n}+\frac{\mathfrak{h}}{2}) - \varepsilon^q \bar{z}(\tau_{n})\biggr) \right| \notag \\
		& \,\, \lesssim \varepsilon^q \left| z(\tau_{n}+\frac{\mathfrak{h}}{2}) -z(\tau_{n})-\frac{1}{2}\varepsilon\mathfrak{h}\varphi_{1}\biggl(\frac{\mathfrak{h}\widehat{B}_{0}}{2}\biggr)w(\tau_{n})  \right| \notag \\
		& \,\, = \varepsilon^q \left| z(\tau_{n}) + \frac{\mathfrak{h}}{2}z'(\tau_{n}) + O(\varepsilon \mathfrak{h}^{2}) - z(\tau_{n}) - \frac{1}{2}\varepsilon\mathfrak{h} \biggl(I + \frac{\mathfrak{h}}{2 \cdot 2!}\widehat{B}_{0} + \cdots \biggr) w(\tau_{n})\right| \notag\\
		& \,\, \lesssim \varepsilon^{q+1} \mathfrak{h}^{2},
		\label{3.17}
	\end{align}
	where $\nabla\widehat{B}$ is the derivative of $\widehat{B}$,
	and consequently leads to
	\begin{equation}
		|\xi_{w,1}^{n}| \lesssim \varepsilon^{q+1} \mathfrak{h}^{3},
		\quad 1 \leq q \leq 2, 
		\quad 0 \leq n < \frac{T}{\varepsilon \mathfrak{h}}.
		\label{3.18}
	\end{equation}
	
	Next, we proceed to analyze $\xi_{w,2}^{n}$.
	Performing integration by substitution on the first term of \eqref{3.16b}, we have
	\begin{equation*}
		\varepsilon \int_{0}^{\mathfrak{h}} \mathrm{e}^{(\mathfrak{h}-s)\widehat{B}_{n}} E(\widetilde{z}^{n}(s)) \, ds
		= \varepsilon \mathfrak{h} \int_{0}^{1} \mathrm{e}^{\rho \mathfrak{h} \widehat{B}_{n}} E(\widetilde{z}^{n}((1-\rho)\mathfrak{h})) \, d\rho.
	\end{equation*}
	By means of \eqref{A} and Taylor expansion, it follows directly that
	\begin{align*}
		& E(\widetilde{z}^{n}((1-\rho)\mathfrak{h})) = E(z(\tau_{n} + (1-\rho)\mathfrak{h}) - \zeta_{z}^{n}((1-\rho)\mathfrak{h}) ) \\
		&\,\, = E \biggl( z(\tau_{n} + \frac{\mathfrak{h}}{2}) + (\frac{1}{2} - \rho)\mathfrak{h} z'(\tau_{\rho}^{n}) - \zeta_{z}^{n}((1-\rho)\mathfrak{h}) \biggr) \\
		&\,\, = E\biggl(z(\tau_{n} + \frac{\mathfrak{h}}{2}) \biggr) + \int_{0}^{1}\nabla E(s_{\sigma}) \, d\sigma \biggl[ (\frac{1}{2} - \rho)\mathfrak{h} \varepsilon w(\tau_{\rho}^{n}) - \zeta_{z}^{n}((1-\rho) \mathfrak{h}) \biggr],
	\end{align*}
	where $\nabla E$ represents the derivative of $E$, $\tau_{\rho}^{n} \in (\tau_{n} + \frac{\mathfrak{h}}{2},\tau_{n}+(1-\rho)\mathfrak{h})$ and
	\begin{equation*}
		s_{\sigma} = z(\tau_{n} + \frac{\mathfrak{h}}{2}) + \sigma \biggl[ (\frac{1}{2} - \rho)\mathfrak{h} \varepsilon w(\tau_{\rho}^{n}) - \zeta_{z}^{n}((1-\rho) \mathfrak{h}) \biggr].
	\end{equation*}
	So we can find that
	\begin{align}
		\xi_{w,2}^{n} = & \varepsilon \int_{0}^{\mathfrak{h}} \mathrm{e}^{(\mathfrak{h}-s)\widehat{B}_{n}} E(\widetilde{z}^{n}(s)) \, ds
		- \varepsilon\mathfrak{h}\mathrm{e}^{\frac{\mathfrak{h}\widehat{B}_{0}}{2}}
		\varphi_{1}\biggl(\mathfrak{h}(\widehat{B}_{\bar{z}(\tau_{n})}-\widehat{B}_{0})\biggr) E(\bar{z}(\tau_{n}))  \notag \\
		= & \varepsilon \mathfrak{h} \int_{0}^{1} \mathrm{e}^{\rho \mathfrak{h} \widehat{B}_{n}}  E\biggl(z(\tau_{n} + \frac{\mathfrak{h}}{2}) \biggr) \, d\rho
		- \varepsilon\mathfrak{h}\mathrm{e}^{\frac{\mathfrak{h}\widehat{B}_{0}}{2}}
		\varphi_{1}\biggl(\mathfrak{h}(\widehat{B}_{\bar{z}(\tau_{n})}-\widehat{B}_{0})\biggr) E(\bar{z}(\tau_{n}))
		\notag \\ 
		& + \varepsilon \mathfrak{h} \int_{0}^{1} \mathrm{e}^{\rho \mathfrak{h} \widehat{B}_{n}}
		\int_{0}^{1}\nabla E(s_{\sigma}) \, d\sigma \biggl[ (\frac{1}{2} - \rho)\mathfrak{h} \varepsilon w(\tau_{\rho}^{n}) - \zeta_{z}^{n}((1-\rho) \mathfrak{h}) \biggr] \, d\rho.
		\label{3.19}
	\end{align}
	Applying the midpoint integral formula to \eqref{3.19} 
	and noting
	\begin{align} 
		& \left| E(z(\tau_{n}+\frac{\mathfrak{h}}{2})) - E(\bar{z}(\tau_{n})) \right| 
		= \left|\nabla E \cdot \biggl(z(\tau_{n}+\frac{\mathfrak{h}}{2}) -  \bar{z}(\tau_{n})\biggr) \right| \notag \\
		& \,\, \lesssim \left| z(\tau_{n}+\frac{\mathfrak{h}}{2}) -z(\tau_{n})-\frac{1}{2}\varepsilon\mathfrak{h}\varphi_{1}\biggl(\frac{\mathfrak{h}\widehat{B}_{0}}{2}\biggr)w(\tau_{n})  \right| \notag \\
		& \,\, = \left| z(\tau_{n}) + \frac{\mathfrak{h}}{2}z'(\tau_{n}) + O(\varepsilon \mathfrak{h}^{2}) - z(\tau_{n}) - \frac{1}{2}\varepsilon\mathfrak{h} \biggl(I + \frac{\mathfrak{h}}{2 \cdot 2!}\widehat{B}_{0} + ... \biggr) w(\tau_{n})\right| \notag\\
		& \,\, \lesssim \varepsilon \mathfrak{h}^{2},
		\label{3.20}
	\end{align}
	the following result can be derived
	\begin{equation}
		|\xi_{w,2}^{n}| \lesssim \varepsilon^{2} \mathfrak{h}^{3}, 
		\quad 0 \leq n < \frac{T}{\varepsilon \mathfrak{h}}.
		\label{3.21}
	\end{equation}
	Combining \eqref{3.18} and \eqref{3.21}, we have 
	\begin{equation*}
		|\xi_{w}^{n}| \lesssim \varepsilon^{2} \mathfrak{h}^{3}, 
		\quad 0 \leq n < \frac{T}{\varepsilon \mathfrak{h}}.
	\end{equation*}
	
	$\bullet$   \textbf{Estimate of $\xi_{z}^{n}$.}  
	By subtracting \eqref{3.15} from \eqref{3.13a}, it can be shown that
	\begin{equation*}
		\xi_{z}^{n} = \xi_{z,1}^{n} + \xi_{z,2}^{n}, 
		\quad 0 \leq n < \frac{T}{\varepsilon \mathfrak{h}},
	\end{equation*}
	with
	\begin{subequations}
		\begin{align}
			\xi_{z,1}^{n} = & \varepsilon \int_{0}^{\mathfrak{h}} \mathrm{e}^{s \widehat{B}(\varepsilon^q z(\tau_{n}+\frac{\mathfrak{h}}{2}))} w(\tau_{n}) \, ds \notag \\
			& -
			\frac{\varepsilon\mathfrak{h}}{2}\varphi_{1}\biggl(\frac{\mathfrak{h}\widehat{B}_{0}}{2}\biggr)\Bigl(I+\mathrm{e}^{\mathfrak{h}(\widehat{B}_{\bar{z}(\tau_{n})}-\widehat{B}_{0})}\mathrm{e}^{\frac{\mathfrak{h}\widehat{B}_{0}}{2}}\Bigr)w(\tau_{n}),
			\label{3.23a} \\
			\xi_{z,2}^{n} = & \varepsilon^{2} \int_{0}^{\mathfrak{h}} \int_{0}^{s} \mathrm{e}^{(s-\sigma)\widehat{B}(\varepsilon^q z(\tau_{n}+\frac{\mathfrak{h}}{2}))} E(\widetilde{z}^{n}(\sigma)) \, d\sigma \, ds \notag\\
			& - \frac{\varepsilon^{2}\mathfrak{h}^{2}}{2}\varphi_{1}\biggl(\frac{\mathfrak{h}\widehat{B}_{0}}{2}\biggr)
			\varphi_{1}\biggl(\mathfrak{h}(\widehat{B}_{\bar{z}(\tau_{n})}-\widehat{B}_{0})\biggr) E(\bar{z}(\tau_{n})).	 \label{3.23b}
		\end{align}
	\end{subequations}
	Applying the first mean value theorem for integrals to \eqref{3.23a} yields
	\begin{equation*}
		\xi_{z,1}^{n} = \varepsilon \mathfrak{h} \mathrm{e}^{\lambda \widehat{B}_{n}} w(\tau_{n}) 
		-\frac{\varepsilon\mathfrak{h}}{2}\varphi_{1}\biggl(\frac{\mathfrak{h}\widehat{B}_{0}}{2}\biggr)\Bigl(I+\mathrm{e}^{\mathfrak{h}(\widehat{B}_{\bar{z}(\tau_{n})}-\widehat{B}_{0})}\mathrm{e}^{\frac{\mathfrak{h}\widehat{B}_{0}}{2}}\Bigr)w(\tau_{n}), 
	\end{equation*}
	where $\lambda \in (0,\mathfrak{h})$. 
	According to the definitions of \(\widehat{B}_{n}\) and \(\widehat{B}_{0}\), we have  
	\begin{equation}
		| \widehat{B}_{n} - \widehat{B}_{0} | \lesssim \varepsilon^{q+1}, 
		\quad | \widehat{B}_{\bar{z}(\tau_{n})} - \widehat{B}_{0} | \lesssim \varepsilon^{q+1},
		\quad 1 \leq q \leq 2.
		\label{Bn}
	\end{equation}
	From \eqref{3.17} and \eqref{Bn}, it can be deduced that
	\begin{equation}
		|\xi_{z,1}^{n}| \lesssim \varepsilon^{q+2} \mathfrak{h}^{3},
		\quad 1 \leq q \leq 2, 
		\quad 0 \leq n < \frac{T}{\varepsilon \mathfrak{h}}.
		\label{3.24}
	\end{equation}
	
	For $\xi_{z,2}^{n}$, performing integration by substitution on the first term of \eqref{3.23b} yields
	\begin{equation*}
		\varepsilon^{2} \int_{0}^{\mathfrak{h}} \int_{0}^{s} \mathrm{e}^{(s-\sigma)\widehat{B}_{n}} E(\widetilde{z}^{n}(\sigma)) \, d\sigma ds
		= \varepsilon^{2} \int_{0}^{\mathfrak{h}} \int_{0}^{1} s \mathrm{e}^{\rho s \widehat{B}_{n}} E(\widetilde{z}^{n}((1-\rho)s)) \, d\rho ds.
	\end{equation*}
	Rewriting the integrand in the above expression as
	\begin{align*}
		& E(\widetilde{z}^{n}((1-\rho)s)) = E(z(\tau_{n} + (1-\rho)s) - \zeta_{z}^{n}((1-\rho)s) )  \\
		&\,\, = E \biggl( z(\tau_{n} + \frac{s}{2}) + (\frac{1}{2} - \rho)s z'(\tau_{\rho}^{n}) - \zeta_{z}^{n}((1-\rho)s) \biggr) \\
		&\,\, = E\biggl(z(\tau_{n} + \frac{s}{2}) \biggr) + \int_{0}^{1}\nabla E(s_{\lambda}) \, d\lambda \biggl[ (\frac{1}{2} - \rho)s \varepsilon w(\tau_{\rho}^{n}) - \zeta_{z}^{n}((1-\rho) s) \biggr],
	\end{align*}
	where $\tau_{\rho}^{n} \in (\tau_{n} + \frac{\mathfrak{h}}{2},\tau_{n}+(1-\rho)s)$,
	\begin{equation*}
		s_{\lambda} = z(\tau_{n} + \frac{s}{2}) + \lambda \biggl[ (\frac{1}{2} - \rho)s \varepsilon w(\tau_{\rho}^{n}) - \zeta_{z}^{n}((1-\rho) s) \biggr],
	\end{equation*}
	then we can derive
	\begin{align*}
		\xi_{z,2}^{n} = & \varepsilon^{2} \int_{0}^{\mathfrak{h}} \int_{0}^{s} \mathrm{e}^{(s-\sigma)\widehat{B}_{n}} E(\widetilde{z}^{n}(\sigma)) \, d\sigma \, ds - \frac{\varepsilon^{2}\mathfrak{h}^{2}}{2}\varphi_{1}\biggl(\frac{\mathfrak{h}\widehat{B}_{0}}{2}\biggr)
		\varphi_{1}\biggl(\mathfrak{h}(\widehat{B}_{\bar{z}(\tau_{n})}-\widehat{B}_{0})\biggr) E(\bar{z}(\tau_{n}))  \\
		= & \varepsilon^{2} \int_{0}^{\mathfrak{h}} \int_{0}^{1} s \mathrm{e}^{\rho s \widehat{B}_{n}} E\biggl(z(\tau_{n} + \frac{s}{2}) \biggr) \, d\rho ds \\ 
		& - \varepsilon^{2} \int_{0}^{\mathfrak{h}} \int_{0}^{1} s \varphi_{1}\biggl(\frac{\mathfrak{h}\widehat{B}_{0}}{2}\biggr)
		\varphi_{1}\biggl(\mathfrak{h}(\widehat{B}_{\bar{z}(\tau_{n})}-\widehat{B}_{0})\biggr) E(\bar{z}(\tau_{n})) d\rho ds \\
		& + \varepsilon^{2} \int_{0}^{\mathfrak{h}} \int_{0}^{1} s \mathrm{e}^{\rho s \widehat{B}_{n}}
		\int_{0}^{1} \nabla E(s_{\lambda}) \, d\lambda \biggl[ (\frac{1}{2} - \rho)s \varepsilon w(\tau_{\rho}^{n}) - \zeta_{z}^{n}((1-\rho) s) \biggr] \, d\rho ds.
	\end{align*}
	Using the mean-value property of the derivative, we get
	\begin{align*} 
		& \left| E(z(\tau_{n}+\frac{s}{2})) - E(\bar{z}(\tau_{n})) \right| 
		= \left|\nabla E \cdot \biggl(z(\tau_{n}+\frac{s}{2}) -  \bar{z}(\tau_{n})\biggr) \right|  \\
		& \,\, \lesssim \left| z(\tau_{n}+\frac{s}{2}) -z(\tau_{n})-\frac{1}{2}\varepsilon\mathfrak{h}\varphi_{1}\biggl(\frac{\mathfrak{h}\widehat{B}_{0}}{2}\biggr)w(\tau_{n})  \right|  \\
		& \,\, = \left| z(\tau_{n}) + \frac{s}{2}z'(\tau_{n}) + O(\varepsilon s^{2}) - z(\tau_{n}) - \frac{1}{2}\varepsilon\mathfrak{h} \biggl(I + \frac{\mathfrak{h}}{2 \cdot 2!}\widehat{B}_{0} + ... \biggr) w(\tau_{n})\right|  \\
		& \,\, \lesssim \varepsilon (s-\mathfrak{h}).
	\end{align*}
	Consequently, the following estimate can be obtained
	\begin{equation}
		|\xi_{z,2}^{n}| \lesssim \varepsilon^{3} \mathfrak{h}^{3}, 
		\quad 0 \leq n < \frac{T}{\varepsilon \mathfrak{h}}.
		\label{3.25}
	\end{equation}
	Finally, by combining \eqref{3.24} and \eqref{3.25}, we obtain
	\begin{equation*}
		|\xi_{z}^{n}| \lesssim \varepsilon^{3} \mathfrak{h}^{3}, 
		\quad 0 \leq n < \frac{T}{\varepsilon \mathfrak{h}}.
	\end{equation*}
	The proof of Lemma \ref{lemmalocal} is complete.
\end{proof}

\subsection{Standard global error estimates}\label{subsec3.4}

Building upon the foregoing preparations, a standard global error estimate and the boundedness of the numerical solution can be deduced.

\begin{lemma}\label{le3.1}
For the numerical solutions \( z^n \) and \( w^n \) obtained by applying the S2‑new scheme to the long-time problem \eqref{3.3} on the time interval \([0,T/\varepsilon]\), there exists a positive constant \( \mathfrak{h}_0 \) independent of \( \varepsilon \), such that when the time step satisfies \( 0 < \mathfrak{h} < \mathfrak{h}_0 \), we have
	\begin{equation*}
		|z^n - z(\tau_n)| \lesssim \varepsilon \mathfrak{h}^2, \quad |w^n - w(\tau_n)| \lesssim \varepsilon \mathfrak{h}^2,
		\quad 0 \leq n \leq \frac{T}{\varepsilon \mathfrak{h}},
	\end{equation*}
	and  
	\begin{equation}
		|z^n| \leq \|z\|_{L^\infty(0,T/\varepsilon)} + 1, \quad |w^n| \leq \|w\|_{L^\infty(0,T/\varepsilon)} + 1, \quad 0 \leq n \leq \frac{T}{\varepsilon \mathfrak{h}}.
		\label{3.27}
	\end{equation}
\end{lemma}

\begin{proof}
	For \( n = 0 \), the initial conditions \( z^0 = x_0 \) and \( w^0 = v_0 \) imply that Lemma \ref{le3.1} is trivially satisfied. Now, assume that for some integer \( m \) with \( 0 \leq m < \frac{T}{\varepsilon \mathfrak{h}} \), the estimate \eqref{3.27} holds. We proceed to show that the same estimate remains valid for \( m + 1 \).
	
	Subtracting \eqref{3.13a}-\eqref{3.13b} from the scheme \eqref{3.5} for \( n \leq m \) and using the relation given in \eqref{B1}-\eqref{B2}, we obtain
	\begin{subequations}\label{3.28}
		\begin{align}
			e_{z}^{n+1} &= e_z^n + \frac{\varepsilon \mathfrak{h}}{2} \varphi_{1}(\frac{\mathfrak{h}}{2}\widehat{B}_{0})
			\Bigl(I+\mathrm{e}^{\mathfrak{h}(\widehat{B}_{\bar{z}(\tau_{n})}-\widehat{B}_{0})}\mathrm{e}^{\frac{\mathfrak{h}}{2}\widehat{B}_{0}}\Bigr) e_w^n 
			+ \eta_z^n + \xi_z^n + \zeta_z^n(\mathfrak{h}), \label{3.28a} \\
			e_{w}^{n+1} &= \mathrm{e}^{\frac{\mathfrak{h}}{2}\widehat{B}_{0}}
			\mathrm{e}^{\mathfrak{h}(\widehat{B}_{\bar{z}(\tau_{n})}-\widehat{B}_{0})}\mathrm{e}^{\frac{\mathfrak{h}}{2}\widehat{B}_{0}} e_w^n
			+ \eta_w^n + \xi_w^n + \zeta_w^n(\mathfrak{h}), \quad 0 \leq n \leq m, \label{3.28b}
		\end{align}
	\end{subequations}
	where 
	\begin{align*}
		\eta_{z}^{n} = & \left[ \frac{\varepsilon \mathfrak{h}}{2} \varphi_{1}(\frac{\mathfrak{h}}{2}\widehat{B}_{0}) \mathrm{e}^{\mathfrak{h}(\widehat{B}_{\bar{z}(\tau_{n})}-\widehat{B}_{0})}
		\mathrm{e}^{\frac{\mathfrak{h}}{2}\widehat{B}_{0}}
		- \frac{\varepsilon \mathfrak{h}}{2} \varphi_{1}(\frac{\mathfrak{h}}{2}\widehat{B}_{0}) \mathrm{e}^{\mathfrak{h}(\widehat{B}_{\bar{z}_{n}}-\widehat{B}_{0})}\mathrm{e}^{\frac{\mathfrak{h}}{2}\widehat{B}_{0}} \right] w^{n}   \\
		& + \frac{\varepsilon^{2}\mathfrak{h}^{2}}{2}\varphi_{1}
		\biggl(\frac{\mathfrak{h}}{2} \widehat{B}_{0} \biggr)
		\varphi_{1}\biggl(\mathfrak{h}(\widehat{B}_{\bar{z}(\tau_{n})}-\widehat{B}_{0})\biggr)
		E(\bar{z}(\tau_{n})) \\
		& - \frac{\varepsilon^{2}\mathfrak{h}^{2}}{2}\varphi_{1}
		\biggl(\frac{\mathfrak{h}}{2} \widehat{B}_{0} \biggr)
		\varphi_{1}\biggl(\mathfrak{h}(\widehat{B}_{\bar{z}_{n}}-\widehat{B}_{0})\biggr)
		E(\bar{z}_{n}),  \\
		\eta_{w}^{n} = & \left[ \mathrm{e}^{\frac{\mathfrak{h}}{2}\widehat{B}_{0}}
		\mathrm{e}^{\mathfrak{h}(\widehat{B}_{\bar{z}(\tau_{n})}-\widehat{B}_{0})}\mathrm{e}^{\frac{\mathfrak{h}}{2}\widehat{B}_{0}}
		- \mathrm{e}^{\frac{\mathfrak{h}}{2}\widehat{B}_{0}}
		\mathrm{e}^{\mathfrak{h}(\widehat{B}_{\bar{z}_{n}}-\widehat{B}_{0})}
		\mathrm{e}^{\frac{\mathfrak{h}}{2}\widehat{B}_{0}} \right] w^{n}  \\
		& + \varepsilon\mathfrak{h}\mathrm{e}^{\frac{\mathfrak{h}}{2}\widehat{B}_{0}}
		\varphi_{1}\biggl(\mathfrak{h}(\widehat{B}_{\bar{z}(\tau_{n})}-\widehat{B}_{0})\biggr) E(\bar{z}(\tau_{n})) - \varepsilon\mathfrak{h}\mathrm{e}^{\frac{\mathfrak{h}}{2}\widehat{B}_{0}}
		\varphi_{1}\biggl(\mathfrak{h}(\widehat{B}_{\bar{z}_{n}}-\widehat{B}_{0})\biggr) E(\bar{z}_{n}).
	\end{align*}
	We first estimate the leading term of $\eta_{z}^{n}$. 
	From the notations of \eqref{C1} and \eqref{C2}, we have
	\begin{align} 
		& \left|  \widehat{B}_{\bar{z}(\tau_{n})} - \widehat{B}_{\bar{z}_{n}} \right| 
		= \left|\nabla\widehat{B} \cdot (\varepsilon^q \bar{z}(\tau_{n}) 
		- \varepsilon^q \bar{z}_{n} ) \right| \notag \\
		& \,\, \lesssim \varepsilon^q \left| z(\tau_{n})+\frac{1}{2}\varepsilon\mathfrak{h}\varphi_{1}
		\biggl(\frac{\mathfrak{h}\widehat{B}_{0}}{2}\biggr)w(\tau_{n}) - z^{n} - \frac{1}{2}\varepsilon\mathfrak{h}\varphi_{1}\biggl(\frac{\mathfrak{h}\widehat{B}_{0}}{2}\biggr)w^{n}  \right| \notag \\
		& \,\, \lesssim \varepsilon^q (|e_{z}^{n}| + \varepsilon \mathfrak{h} |e_{w}^{n}|),
		\quad 1 \leq q \leq 2,
		\label{3.29}
	\end{align}
	thus the leading-order estimate for $\eta_{z}^{n}$ is as follows
	\begin{align*}
		& \left| \left[ \frac{\varepsilon \mathfrak{h}}{2} \varphi_{1}(\frac{\mathfrak{h}}{2}\widehat{B}_{0}) \mathrm{e}^{\mathfrak{h}(\widehat{B}_{\bar{z}(\tau_{n})}-\widehat{B}_{0})}
		\mathrm{e}^{\frac{\mathfrak{h}}{2}\widehat{B}_{0}}
		- \frac{\varepsilon \mathfrak{h}}{2} \varphi_{1}(\frac{\mathfrak{h}}{2}\widehat{B}_{0}) \mathrm{e}^{\mathfrak{h}(\widehat{B}_{\bar{z}_{n}}-\widehat{B}_{0})}\mathrm{e}^{\frac{\mathfrak{h}}{2}\widehat{B}_{0}} \right] w^{n} \right|   \\
		& \,\, \lesssim \varepsilon \mathfrak{h}^{2} \cdot \varepsilon^q (|e_{z}^{n}| + \varepsilon \mathfrak{h} |e_{w}^{n}|)  \\
		& \,\, \lesssim \varepsilon^{q+1} \mathfrak{h}^{2} |e_{z}^{n}| + 
		\varepsilon^{q+2} \mathfrak{h}^{3} |e_{w}^{n}|,
		\quad 1 \leq q \leq 2. 
	\end{align*}
	Meanwhile, by observing that $1 \leq q \leq 2$ and
	\begin{equation}
		\left| E(\bar{z}(\tau_{n})) - E(\bar{z}_{n}) \right| \lesssim 
		|e_{z}^{n}| + \varepsilon \mathfrak{h} |e_{w}^{n}|,
		\label{En}
	\end{equation}
	the estimates for the remaining two terms of $\eta_{z}^{n}$ are derived as
	\begin{align*}
		&\left| \frac{\varepsilon^{2}\mathfrak{h}^{2}}{2}\varphi_{1}
		\biggl(\frac{\mathfrak{h}}{2} \widehat{B}_{0} \biggr)
		\varphi_{1}\biggl(\mathfrak{h}(\widehat{B}_{\bar{z}(\tau_{n})}-\widehat{B}_{0})\biggr)
		E(\bar{z}(\tau_{n}))
		- \frac{\varepsilon^{2}\mathfrak{h}^{2}}{2}\varphi_{1}
		\biggl(\frac{\mathfrak{h}}{2} \widehat{B}_{0} \biggr)
		\varphi_{1}\biggl(\mathfrak{h}(\widehat{B}_{\bar{z}_{n}}-\widehat{B}_{0})\biggr)
		E(\bar{z}_{n}) \right| \\
		& \,\, = \left| \frac{\varepsilon^{2}\mathfrak{h}^{2}}{2}\varphi_{1}
		\biggl(\frac{\mathfrak{h}}{2} \widehat{B}_{0} \biggr)
		\varphi_{1}\biggl(\mathfrak{h}(\widehat{B}_{\bar{z}(\tau_{n})}-\widehat{B}_{0})\biggr)
		[E(\bar{z}(\tau_{n})) - E(\bar{z}_{n})] \right|  \\
		& \,\,\,\,\,\,\,\,\, + \left| \frac{\varepsilon^{2}\mathfrak{h}^{2}}{2}\varphi_{1}
		\biggl(\frac{\mathfrak{h}}{2} \widehat{B}_{0} \biggr)
		\left[ \varphi_{1}\biggl(\mathfrak{h}(\widehat{B}_{\bar{z}(\tau_{n})}-\widehat{B}_{0})\biggr)
		- \varphi_{1}\biggl(\mathfrak{h}(\widehat{B}_{\bar{z}_{n}}-\widehat{B}_{0})\biggr)
		\right] E(\bar{z}_{n}) \right| \\
		& \,\, \lesssim \varepsilon^{2} \mathfrak{h}^{2} 
		\left| E(\bar{z}(\tau_{n})) - E(\bar{z}_{n}) \right|
		+ \varepsilon^{2} \mathfrak{h}^{2} \cdot \mathfrak{h} 
		|\widehat{B}_{\bar{z}(\tau_{n})} - \widehat{B}_{\bar{z}_{n}}|   \\
		& \,\, \lesssim \varepsilon^{2} \mathfrak{h}^{2} 
		(|e_{z}^{n}| + \varepsilon \mathfrak{h} |e_{w}^{n}|)
		+ \varepsilon^{2} \mathfrak{h}^{3} \cdot \varepsilon^q
		(|e_{z}^{n}| + \varepsilon \mathfrak{h} |e_{w}^{n}|)  \\
		& \,\, \lesssim \varepsilon^{2} \mathfrak{h}^{2} |e_{z}^{n}| + 
		\varepsilon^{3} \mathfrak{h}^{3} |e_{w}^{n}|. 
	\end{align*}
	Based on the above derivation, we obtain
	\begin{equation}
		|\eta_{z}^{n}| \lesssim \varepsilon^{2} \mathfrak{h}^{2} |e_{z}^{n}| + 
		\varepsilon^{3} \mathfrak{h}^{3} |e_{w}^{n}|, \quad 0 \leq n \leq m.
		\label{3.30}
	\end{equation}
	Similarly, from \eqref{3.29}, \eqref{En} and $1 \leq q \leq 2$, it can be straightforwardly deduced that
	\begin{equation}
		|\eta_{w}^{n}| \lesssim \varepsilon \mathfrak{h} |e_{z}^{n}| + 
		\varepsilon^{2} \mathfrak{h}^{2} |e_{w}^{n}|, \quad 0 \leq n \leq m.
		\label{3.31}
	\end{equation}
	
	Taking absolute values on both sides of \eqref{3.28a} and \eqref{3.28b} and applying the triangle inequality, along with the orthogonality property of the matrix $\mathrm{e}^{\mathfrak{h} \widehat{B}}$, we derive 
	\begin{align*}
		|e_{z}^{n+1}| & < |e_{z}^{n}|+\varepsilon{\mathfrak{h}}|e_{w}^{n}|+|\eta _{z}^{n}|+|\xi_{z}^{n}|+|\zeta_{z}^{n}({\mathfrak{h}})|,   \\
		|e_{w}^{n+1}| & <|e_{w}^{n}|
		+|\eta_{w}^{n}|+|\xi_{w}^{n}|+|\zeta_{w}^{n}({\mathfrak{h}})|,
		\quad 0\leq n\leq m.
	\end{align*}
	Adding these inequalities together and using \eqref{3.30}-\eqref{3.31}, it readily follows that
	\begin{equation*}
		|e_{z}^{n+1}|+|e_{w}^{n+1}|-|e_{z}^{n}|-|e_{w}^{n}|
		\lesssim \varepsilon{\mathfrak{h}} (|e_{z}^{n}|+|e_{w}^{n}|)+|\xi_{z}^{n}|+|\xi_{w}^{n}|+|\zeta_{z}^{n}(\mathfrak{h})|
		+|\zeta_{w}^{n}(\mathfrak{h})|, \quad 0\leq n\leq m.
	\end{equation*}
	Summing the above inequalities over \(0 \leq n \leq m\) and noting that \(e_{z}^{0} = e_{w}^{0} = 0\), we obtain
	\begin{equation*}
		|e_{z}^{m+1}|+|e_{w}^{m+1}| \lesssim \varepsilon {\mathfrak{h}} \sum_{n=0}^{m}(|e_{z}^{n}|+|e_{w}^{n}|)
		+\sum_{n=0}^{m}(|\xi_{z}^{n}|+|\xi_{w}^{n}|+|\zeta_{z}^{n}(\mathfrak{h})|
		+|\zeta_{w}^{n}(\mathfrak{h})|). 
	\end{equation*}
	From the error estimates in \eqref{3.10}, \eqref{3.11} and \eqref{local}, together with the conditions $1 \leq q \leq 2$ and  $m\varepsilon{\mathfrak{h}}\lesssim 1$, we can get
	\begin{equation*}
		|e_{z}^{m+1}|+|e_{w}^{m+1}| \lesssim \varepsilon {\mathfrak{h}} \sum_{n=0}^{m}(|e_{z}^{n}|+|e_{w}^{n}|) + \varepsilon {\mathfrak{h}}^{2}, 
		\quad 0\leq m < \frac{T}{\varepsilon{\mathfrak{h}}}, 
	\end{equation*}
	and applying Gronwall's inequality to the above inequality yields
	\begin{equation*}
		|e_{z}^{m+1}|+|e_{w}^{m+1}| \lesssim \varepsilon {\mathfrak{h}}^{2}, 
		\quad 0\leq m < \frac{T}{\varepsilon{\mathfrak{h}}}.
	\end{equation*}
	Since
	\begin{equation*}
		|z^{m+1}|\leq|z(\tau_{m+1})|+|e_{z}^{m+1}|,
		\quad |w^{m+1}|\leq|w(\tau_{m+1})|+|e_{w}^{m+1}|,
		\quad 0\leq m < \frac{T}{\varepsilon{\mathfrak{h}}},  
	\end{equation*}
	we can choose a positive constant \(\mathfrak{h}_{0}\) independent of both \(\varepsilon\) and \(m\), such that for \(0 < \mathfrak{h} \leq \mathfrak{h}_{0}\), the estimate \eqref{3.27} holds for \(m+1\). This completes the induction step and concludes the proof of the convergence Lemma \ref{le3.1}. 
\end{proof}

\subsection{Improved error estimate}\label{subsec3.5}

The aforementioned standard error analysis fails to yield the optimal error bound for the scheme. To obtain an improved estimate, it is necessary to exploit the skew-symmetry of the magnetic field to construct a propagator that generates a periodic flow. Based on the standard error analysis  and by carefully analyzing the error propagation over each period over long times, we   improve the error bound from $\mathcal{O}(\varepsilon \mathfrak{h}^{2})$ to $\mathcal{O}(\varepsilon^q \mathfrak{h}^{2})$ with $1 \leq q \leq 2$.

\begin{lemma}\label{le3.2}
Under the conditions of \( E(\cdot), B(\cdot) \in C^2(\mathbb{R}^3) \) and Lemma \ref{le3.1}, and let \(T_0 > 0\) be the single time period in the periodic flow generated by \(e^{\tau\widehat{B}(0)}\).  There exists a positive constant $N_0$ independent of $\varepsilon$, such that for any integer $N > N_0$ with the time step $\mathfrak{h} = T_0/N$, the following error bounds hold:
	\begin{equation}
		|z^n - z(\tau_n)| \lesssim \varepsilon^q \mathfrak{h}^{2}, 
		\quad |w_{\parallel}^n - w_{\parallel}(\tau_n)| \lesssim \varepsilon^q \mathfrak{h}^{2}, 
		\quad 1 \leq q \leq 2,
		\quad 0 \leq n \leq \frac{T}{\varepsilon \mathfrak{h}}.
		\label{3.32}
	\end{equation}
	The definitions of \(w_{\parallel}\) and \(w^n_{\parallel}\) are presented below 
	\begin{align*}
		w_{\parallel}(\tau) & := \frac{B(\varepsilon^q z(\tau))}{|B(\varepsilon^q z(\tau))|} \left( \frac{B(\varepsilon^q z(\tau))}{|B(\varepsilon^q z(\tau))|} \cdot w(\tau) \right), \\
		w_{\parallel}^n & := \frac{B(\varepsilon^q z^n)}{|B(\varepsilon^q z^n)|} \left( \frac{B(\varepsilon^q z^n)}{|B(\varepsilon^q z^n)|} \cdot w^n \right).
	\end{align*}
\end{lemma}

\begin{proof}
	Due to the skew-symmetry of \(\widehat{B}\), the propagator \(e^{\tau\widehat{B}(0)}\) generates a periodic flow, and let \(T_0 > 0\) be its single period. Hence, for any fixed $T > 0$, we can rewrite it as  
	\begin{equation*}
		\frac{T}{\varepsilon} = T_0 M + \tau_r, \quad 0 \leq \tau_r < T_0,
	\end{equation*}
	where $M$ is an integer given by  
	\begin{equation*}
		M = \bigg\lfloor  \frac{T}{\varepsilon T_0} \bigg\rfloor  = O(1/\varepsilon).
	\end{equation*}
	Without loss of generality and for simplicity of the analysis, only the case $\tau_r = 0$ will be considered in the following discussion.
	That is, the long-time $T/\varepsilon$ is divided into $M$ periods each of length $T_0$.
	
	First, select a positive constant \(N_0\) such that when \(N \geq N_0\), we have \(0 < \mathfrak{h} = T_0 / N \leq \mathfrak{h}_0 = T_0 / N_0\). This evidently satisfies the condition of Lemma \ref{le3.1}, so the boundedness results in \eqref{3.27} regarding the numerical solution still hold. 
	Then, we update the notation by using $\tau_n^m$ to denote the time grid points within the $m$th period, i.e.,
	\begin{equation*}
		\tau_n^m = mT_0 + n \mathfrak{h}, 
		\quad 0 \leq n \leq N,
		\quad 0 \leq m < M.
	\end{equation*}
	The numerical solutions and errors of scheme \eqref{3.5} on the updated time grid points $\tau_n^m$ are defined as follows
	\begin{equation*}
		z_n^m \approx z(\tau_n^m), \quad w_n^m \approx w(\tau_n^m), 
		\quad 0 \leq n \leq N,
		\quad 0 \leq m < M,
	\end{equation*}
	\begin{equation*}
		e_{z}^{n,m} = z(\tau_{n}^{m}) - z_{n}^{m}, \quad e_{w}^{n,m} = w(\tau_{n}^{m}) - w_{n}^{m}, 
		\quad 0 \leq n \leq N,
		\quad 0 \leq m < M.
	\end{equation*}
	Observing our notations, it is straightforward to see that \(e_{z}^{0,m+1} = e_{z}^{N,m}\) and \(e_{w}^{0,m+1} = e_{w}^{N,m}\).  
	Correspondingly, by analogy with the error expressions \eqref{3.28a}-\eqref{3.28b}, we can derive
	\begin{subequations}\label{3.33}
		\begin{align}
			e_{z}^{n+1,m} = & e_z^{n,m} + \frac{\varepsilon \mathfrak{h}}{2} \varphi_{1}(\frac{\mathfrak{h}}{2}\widehat{B}_{0})
			\Bigl(I+\mathrm{e}^{\mathfrak{h}(\widehat{B}_{\bar{z}(\tau_{n}^{m})}-\widehat{B}_{0})}\mathrm{e}^{\frac{\mathfrak{h}}{2}\widehat{B}_{0}}\Bigr) e_w^{n,m} + \eta_z^{n,m} + \xi_z^{n,m} + \zeta_z^{n,m}(\mathfrak{h}), \label{3.33a} \\
			e_{w}^{n+1,m} = & \mathrm{e}^{\frac{\mathfrak{h}}{2}\widehat{B}_{0}}
			\mathrm{e}^{\mathfrak{h}(\widehat{B}_{\bar{z}(\tau_{n}^{m})}-\widehat{B}_{0})}\mathrm{e}^{\frac{\mathfrak{h}}{2}\widehat{B}_{0}} e_w^{n,m} 
			+ \eta_w^{n,m} + \xi_w^{n,m} + \zeta_w^{n,m}(\mathfrak{h}), \quad 0 \leq n \leq N-1,
			\quad 0 \leq m <M.
			\label{3.33b}
		\end{align}
	\end{subequations}
	
	In order to obtain the improved error bound, we now analyze how the error propagates across each period. For a fixed $m$, summing \eqref{3.33a} over $n$ from $0$ to $N-1$ gives
	\begin{align}
		e_{z}^{N,m} = & e_z^{0,m} + \sum_{n=0}^{N-1} \left[
		\frac{\varepsilon \mathfrak{h}}{2} \varphi_{1}(\frac{\mathfrak{h}}{2}\widehat{B}_{0})
		\Bigl(I+\mathrm{e}^{\mathfrak{h}(\widehat{B}_{\bar{z}(\tau_{n}^{m})}-\widehat{B}_{0})}\mathrm{e}^{\frac{\mathfrak{h}}{2}\widehat{B}_{0}}\Bigr) \right] e_w^{n,m} + \sum_{n=0}^{N-1} (\eta_z^{n,m} + \xi_z^{n,m} + \zeta_z^{n,m}(\mathfrak{h}) )
		\notag \\
		= & e_z^{0,m} + \varepsilon \mathfrak{h} \sum_{n=0}^{N-1} \mathrm{e}^{\frac{\mathfrak{h}}{2} \widehat{B}_{00}} e_{w}^{n,m}
		+ \sum_{n=0}^{N-1} (\eta_z^{n,m} + \xi_z^{n,m} + \zeta_z^{n,m}(\mathfrak{h}) ) +\delta_{z}^{n,m},
		\quad 0 \leq m < M,
		\label{3.34}
	\end{align}
	where $\widehat{B}_{00} = \widehat{B}(0)$ and 
	\begin{equation*}
		\delta_{z}^{n,m} = \varepsilon \mathfrak{h} \sum_{n=0}^{N-1}
		\left[
		\frac{1}{2} \varphi_{1}(\frac{\mathfrak{h}}{2}\widehat{B}_{0})
		\Bigl(I+\mathrm{e}^{\mathfrak{h}(\widehat{B}_{\bar{z}(\tau_{n}^{m})}-\widehat{B}_{0})}\mathrm{e}^{\frac{\mathfrak{h}}{2}\widehat{B}_{0}}\Bigr) - \mathrm{e}^{\frac{\mathfrak{h}}{2} \widehat{B}_{00}} \right] e_w^{n,m}.
	\end{equation*}
	Since \(B(\cdot) \in C^2(\mathbb{R}^3)\) and \(1 \leq q \leq 2\), it is readily obtained that
	\begin{equation}
		| B(\varepsilon^q z(\tau)) - B(0) | \lesssim \varepsilon^q, 
		\quad 1 \leq q \leq 2,
		\quad 0 \leq \tau \leq \frac{T}{\varepsilon}.
		\label{3.35}
	\end{equation}
	From \eqref{3.35} and the estimate \( | e_{w}^{n,m} | \lesssim \varepsilon \mathfrak{h}^{2} \) in Lemma \ref{le3.1}, it follows that
	\begin{equation}
		| \delta_{z}^{n,m} | \lesssim \varepsilon^{q+2} \mathfrak{h}^{3}, 
		\quad 1 \leq q \leq 2,
		\quad 0 \leq m < M.
		\label{3.36}
	\end{equation}
	
	On the other hand, \eqref{3.33b} can be written as 
	\begin{align}
		e_{w}^{n+1,m} = \mathrm{e}^{\mathfrak{h} \widehat{B}_{00}}  e_{w}^{n,m}
		+ \eta_w^{n,m} + \xi_w^{n,m} + \zeta_w^{n,m}(\mathfrak{h}) + \delta_{w}^{n,m}, 
		\quad 0 \leq n \leq N-1,
		\quad 0 \leq m < M.
		\label{3.37}
	\end{align}
	where
	\begin{equation*}
		\delta_{w}^{n,m} = \left[ \mathrm{e}^{\frac{\mathfrak{h}}{2}\widehat{B}_{0}}
		\mathrm{e}^{\mathfrak{h}(\widehat{B}_{\bar{z}(\tau_{n}^{m})}-\widehat{B}_{0})}
		\mathrm{e}^{\frac{\mathfrak{h}}{2}\widehat{B}_{0}} 
		- \mathrm{e}^{\mathfrak{h} \widehat{B}_{00}} \right] e_{w}^{n,m}.
	\end{equation*}
	Similarly, based on \eqref{3.35} and Lemma \ref{le3.1}, we derive
	\begin{equation}
		| \delta_{w}^{n,m} | \lesssim \varepsilon^{q+1} \mathfrak{h}^{3}, 
		\quad 1 \leq q \leq 2,
		\quad 0 \leq m < M.
		\label{3.38}
	\end{equation}
	Recursively from \eqref{3.37}, we find for any $1 \leq n \leq N$ and $0 \leq m < M$, there holds
	\begin{equation*}
		e_{w}^{n,m} = \mathrm{e}^{n \mathfrak{h} \widehat{B}_{00}}  e_{w}^{0,m}
		+ \sum_{j=0}^{n-1} \mathrm{e}^{(n-1-j) \mathfrak{h} \widehat{B}_{00}} 
		\left[ \eta_w^{j,m} + \xi_w^{j,m} + \zeta_w^{j,m}(\mathfrak{h}) + \delta_{w}^{j,m} \right] ,   \\
	\end{equation*}
	from which it follows that
	\begin{align}
		\varepsilon \mathfrak{h} \sum_{n=0}^{N-1} \mathrm{e}^{\frac{\mathfrak{h}}{2} \widehat{B}_{00}} e_{w}^{n,m} = & \varepsilon \mathfrak{h} \sum_{n=0}^{N-1}
		\mathrm{e}^{(n+\frac{1}{2}) \mathfrak{h} \widehat{B}_{00}} e_{w}^{0,m} \notag \\
		& + \varepsilon \mathfrak{h} \sum_{n=0}^{N-1} \sum_{j=0}^{n-1} 
		\mathrm{e}^{(n-\frac{1}{2}-j) \mathfrak{h} \widehat{B}_{00}}
		\left[ \eta_w^{j,m} + \xi_w^{j,m} + \zeta_w^{j,m}(\mathfrak{h}) + \delta_{w}^{j,m} \right].
		\label{3.39}
	\end{align}
	Substituting \eqref{3.39} into the right side of \eqref{3.34} yields
	\begin{equation}
		e_{z}^{N,m} = e_{z}^{0,m} + \varepsilon \mathfrak{h} \sum_{n=0}^{N-1}
		\mathrm{e}^{(n+\frac{1}{2}) \mathfrak{h} \widehat{B}_{00}} e_{w}^{0,m} + \Gamma^{m},
		\quad 0 \leq m < M,
		\label{3.40}
	\end{equation}
	where
	\begin{align*}
		\Gamma^{m} := & \sum_{n=0}^{N-1}
		[ \eta_z^{n,m} + \xi_z^{n,m} + \zeta_z^{n,m}(\mathfrak{h}) ] + \delta_{z}^{n,m}   \\
		& + \varepsilon \mathfrak{h} \sum_{n=0}^{N-1} \sum_{j=0}^{n-1} 
		\mathrm{e}^{(n-\frac{1}{2}-j) \mathfrak{h} \widehat{B}_{00}}
		\left[ \eta_w^{j,m} + \xi_w^{j,m} + \zeta_w^{j,m}(\mathfrak{h}) + \delta_{w}^{j,m} \right].  
	\end{align*}	
	Substituting the estimates from \eqref{3.10}, \eqref{local}, \eqref{3.31}, \eqref{3.38}, and Lemma \ref{le3.1} into the last term of the above expression yields
	\begin{align*}
		\left| \varepsilon \mathfrak{h} \sum_{n=0}^{N-1} \sum_{j=0}^{n-1} 
		\mathrm{e}^{(n-\frac{1}{2}-j) \mathfrak{h} \widehat{B}_{00}}
		\left[ \eta_w^{j,m} + \xi_w^{j,m} + \zeta_w^{j,m}(\mathfrak{h}) + \delta_{w}^{j,m} \right] \right| \lesssim \varepsilon^{3} \mathfrak{h}^{2}
		+ \varepsilon^{2} \mathfrak{h} \sum_{n=0}^{N-1} | e_{z}^{n,m} |,  
	\end{align*}
	and combining with \eqref{3.11}, \eqref{local}, \eqref{3.30} and \eqref{3.36}, we can deduce that
	\begin{equation}
		| \Gamma^{m} | \lesssim \varepsilon^{3} \mathfrak{h}^{2}
		+ \varepsilon^{2} \mathfrak{h} \sum_{n=0}^{N-1} | e_{z}^{n,m} |, 
		\quad 0 \leq m < M. 
		\label{D}
	\end{equation}
	It is noteworthy that since \(1 \leq q \leq 2\), \(| \Gamma^{m} |\) is controlled by the local truncation errors \eqref{local}.
	Using \eqref{D} and considering the quadrature error of the trapezoidal rule, we deduce from \eqref{3.40} that
	\begin{equation}
		| e_{z}^{N,m} | - | e_{z}^{0,m} | \lesssim \varepsilon \left| \int_{0}^{T_0} \mathrm{e}^{s\widehat{B}_{00}} ds \cdot e_{w}^{0,m} \right|
		+ \varepsilon^{3} \mathfrak{h}^{2}
		+ \varepsilon^{2} \mathfrak{h} \sum_{n=0}^{N-1} | e_{z}^{n,m} |, 
		\quad 0 \leq m < M.
		\label{3.41}
	\end{equation}
	Applying the Rodrigues formula yields
	\begin{equation*}
		\mathrm{e}^{s\widehat{B}_{00}} e_{w}^{0,m} = \cos(s|B(0)|) e_{w}^{0,m} + \sin(s|B(0)|) e_{w}^{0,m} \times \widetilde{B}_{00} + (1 - \cos(s|B(0)|)) \left( \widetilde{B}_{00} \cdot e_{w}^{0,m} \right) \widetilde{B}_{00},
	\end{equation*}
	with \(\widetilde{B}_{00} = B(0)/|B(0)|\) denoting the unit magnetic field vector at the origin.
	When integrated over a full period $T_0$, the expression simplifies to retain only
	\begin{equation*}
		\int_{0}^{T_0} \mathrm{e}^{s\widehat{B}_{00}} ds \cdot e_{w}^{0,m} = T_0 
		\left( \widetilde{B}_{00} \cdot e_{w}^{0,m} \right) \widetilde{B}_{00}.
	\end{equation*}
	Consequently, from \eqref{3.41} it follows that
	\begin{align}
		| e_{z}^{N,m} | - | e_{z}^{0,m} | & \lesssim \varepsilon \left| \left( \widetilde{B}_{00} \cdot e_{w}^{0,m} \right) \widetilde{B}_{00} \right|
		+ \varepsilon^{3} \mathfrak{h}^{2}
		+ \varepsilon^{2} \mathfrak{h} \sum_{n=0}^{N-1} | e_{z}^{n,m} | \notag \\
		& \lesssim \varepsilon | e_{w,\parallel}^{0,m} |
		+ \varepsilon^{3} \mathfrak{h}^{2}
		+ \varepsilon^{2} \mathfrak{h} \sum_{n=0}^{N-1} | e_{z}^{n,m} |,
		\quad 0 \leq m < M, 
		\label{3.42}
	\end{align}
	where $e_{w,\parallel}^{0,m}$ is denoted the error $e_{w}^{0,m}$ in the parallel direction of the magnetic field $B(\varepsilon^q z(mT_0))$, i.e.,
	\begin{align*}
		e_{w,\parallel}^{n,m} & := \left( \widetilde{B}^{n,m} \cdot e_{w}^{n,m} \right) \widetilde{B}^{n,m}, 
		\quad 0 \leq n \leq N, \quad 0 \leq m < M, \\
		\widetilde{B}^{n,m} & := \frac{B( \varepsilon^q z(\tau_{n}^{m}))}{|B( \varepsilon^q z(\tau_{n}^{m}))|},
		\quad 1 \leq q \leq 2.
	\end{align*}
	
	Directly applying the error results of Lemma \ref{le3.1} to \eqref{3.33a},
	it readily follows that
	\begin{equation}
		|e_{z}^{n,m} | \lesssim \varepsilon^{2} \mathfrak{h}^{2} + | e_{z}^{0,m} |,
		\quad 1 \leq n \leq N, \quad 0 \leq m < M.
		\label{3.43}
	\end{equation}
	Substituting \eqref{3.43} into \eqref{3.42}, and noting that $e_{z}^{N,m} = e_{z}^{0,m+1}$,  we get 
	\begin{align}
		| e_{z}^{0,m+1} | - | e_{z}^{0,m} | & \lesssim \varepsilon | e_{w,\parallel}^{0,m} |
		+ \varepsilon^{3} \mathfrak{h}^{2}
		+ \varepsilon^{2} ( \varepsilon^{2} \mathfrak{h}^{2} + | e_{z}^{0,m} | ) \notag \\
		& \lesssim \varepsilon | e_{w,\parallel}^{0,m} |
		+ \varepsilon^{3} \mathfrak{h}^{2}
		+ \varepsilon^{2}  | e_{z}^{0,m} |,
		\quad 0 \leq m < M.
		\label{3.44}
	\end{align}
	Taking the inner product of both sides of \eqref{3.33b} with the unit vector $\widetilde{B}^{n+1,m}$ yields
	\begin{align}
		\left| e_{w,\parallel}^{n+1,m} \right| \leq & \left| \widetilde{B}^{n+1,m} \cdot 
		\left( \mathrm{e}^{\frac{\mathfrak{h}}{2} \widehat{B}_{0}}  \mathrm{e}^{\mathfrak{h}(\widehat{B}_{\bar{z}(\tau_{n}^{m})}-\widehat{B}_{0})}\mathrm{e}^{\frac{\mathfrak{h}}{2}\widehat{B}_{0}} e_{w}^{n,m} \right) \right| \notag \\
		& + \left| \widetilde{B}^{n+1,m} \cdot \eta_{w}^{n,m} \right|
		+ \left| \widetilde{B}^{n+1,m} \cdot \xi_{w}^{n,m} \right|
		+ \left| \widetilde{B}^{n+1,m} \cdot \zeta_{w}^{n,m}(\mathfrak{h}) \right|.
		\label{3.45}
	\end{align}
	Observing that 
	\begin{equation*}
		\widetilde{B}^{n+1,m} = \widetilde{B}^{n,m} + O(\varepsilon^{q+1} \mathfrak{h}),
		\quad 1 \leq q \leq 2,
	\end{equation*}
	and employing the Rodrigues formula, we get
	\begin{equation*}
		\left| \widetilde{B}^{n+1,m} \cdot 
		\left( \mathrm{e}^{\frac{\mathfrak{h}}{2} \widehat{B}_{0}}  \mathrm{e}^{\mathfrak{h}(\widehat{B}_{\bar{z}(\tau_{n}^{m})}-\widehat{B}_{0})}\mathrm{e}^{\frac{\mathfrak{h}}{2}\widehat{B}_{0}} e_{w}^{n,m} \right) \right|
		\lesssim | e_{w,\parallel}^{n,m} | + \varepsilon^{q+2} \mathfrak{h}^{3},
		\quad 1 \leq q \leq 2.
	\end{equation*}
	Recalling \eqref{3.16a}-\eqref{3.16b} and applying Rodrigues formula, it follows directly that
	\begin{equation*}
		\left| \widetilde{B}^{n+1,m} \cdot \xi_{w}^{n,m} \right| \lesssim \varepsilon^{q+2} \mathfrak{h}^{3},
		\quad 1 \leq q \leq 2.
	\end{equation*}
	Then, using the triangle inequality to \eqref{3.31} and \eqref{3.10} yields
	\begin{align*}
		\left| \widetilde{B}^{n+1,m} \cdot \eta_{w}^{n,m} \right| 
		& \lesssim \left| \eta_{w}^{n,m} \right|
		\lesssim \varepsilon \mathfrak{h} | e_{z}^{n,m} | 
		+ \varepsilon^{2} \mathfrak{h}^{2} | e_{w}^{n,m} |,   \\
		\left| \widetilde{B}^{n+1,m} \cdot \zeta_{w}^{n,m}(\mathfrak{h}) \right|  & \lesssim \left| \zeta_{w}^{n,m}(\mathfrak{h}) \right| 
		\lesssim \varepsilon^{q+1} \mathfrak{h}^{3},
		\quad 1 \leq q \leq 2.   
	\end{align*}
	Substituting all the above estimates into \eqref{3.45}, we arrive at
	\begin{equation}
		\left| e_{w,\parallel}^{n+1,m} \right| - \left| e_{w,\parallel}^{n,m} \right| \lesssim
		\varepsilon \mathfrak{h} | e_{z}^{n,m} | + \varepsilon^{q+1} \mathfrak{h}^{3},
		\quad 1 \leq q \leq 2,
		\quad 0 \leq n \leq N-1, \quad 0 \leq m <M.
		\label{3.46}
	\end{equation} 
	Summing \eqref{3.46} over $n$ from $0$ to $N-1$ and using the relation $e_{w}^{N,m} = e_{w}^{0,m+1}$, we can derive that
	\begin{equation*}
		\left| e_{w,\parallel}^{0,m+1} \right| - \left| e_{w,\parallel}^{0,m} \right| \lesssim
		\varepsilon \mathfrak{h} \sum_{n=0}^{N-1} | e_{z}^{n,m} | + \sum_{n=0}^{N-1} \varepsilon^{q+1} \mathfrak{h}^{3},
		\quad 1 \leq q \leq 2,
		\quad 0 \leq m <M.
	\end{equation*}
	Plugging \eqref{3.43} into the above inequality leads to
	\begin{equation}
		\left| e_{w,\parallel}^{0,m+1} \right| - \left| e_{w,\parallel}^{0,m} \right| \lesssim
		\varepsilon  | e_{z}^{0,m} | + \varepsilon^{q+1} \mathfrak{h}^{2},
		\quad 1 \leq q \leq 2,
		\quad 0 \leq m < M.
		\label{3.47}
	\end{equation}
	Finally, adding \eqref{3.44} and \eqref{3.47} together yields
	\begin{equation*}
		\left| e_{z}^{0,m+1} \right| + \left| e_{w,\parallel}^{0,m+1} \right| 
		- \left| e_{z}^{0,m} \right| - \left| e_{w,\parallel}^{0,m} \right| \lesssim 
		\varepsilon ( | e_{z}^{0,m} | +  | e_{w,\parallel}^{0,m} | ) 
		+ \varepsilon^{q+1} \mathfrak{h}^{2}, 
		\quad 1 \leq q \leq 2,
		\quad 0 \leq m < M.
	\end{equation*}
	Applying Gronwall's inequality to the above expression and noting the initial error conditions \(e_{z}^{0,0} = e_{w,\parallel}^{0,0} = 0\), we obtain
	\begin{equation}
		\left| e_{z}^{0,m} \right| + \left| e_{w,\parallel}^{0,m} \right| \lesssim 
		\varepsilon^q \mathfrak{h}^{2}, 
		\quad 1 \leq  q \leq 2,
		\quad  0 \leq m \leq M. 
		\label{E}
	\end{equation}
	By substituting \eqref{E} into \eqref{3.43} and \eqref{3.46}, we immediately obtain the estimates for \(e_{z}^{n,m}\) and \(e_{w,\parallel}^{n,m}\) with \(0 < n < N\).
	This completes the proof of Lemma \ref{le3.2}.
\end{proof} 

To return to the original scale, set the time step size \( h = \varepsilon \mathfrak{h} \). Then we have
\begin{equation*}
x(t_n) = z(\tau_n),\quad v(t_n) = w(\tau_n),\quad x^n = z^n,\quad v^n = w^n, \quad n \geq 0.
\end{equation*}
By defining \( v_{\parallel}(t) := w_{\parallel}(\tau) \) and \( v_{\parallel}^n := w_{\parallel}^n \), we obtain the convergence result of the S2-new scheme for the original system \eqref{3.1}, i.e., equation \eqref{2.7} in Theorem \ref{th2.1} holds. The proof is now complete.

\section{Conclusion}\label{sec4}

In this work, we investigate the numerical approximation of the CPD involving a small parameter $0<\varepsilon \ll 1$, which characterizes the strength of the external magnetic field. We propose a novel splitting approach and construct an explicit symmetric splitting scheme. For the case $1 \leq q \leq 2$ in \eqref{1.1}, our scheme achieves error bounds superior to those of traditional second-order splitting methods. Specifically, in the regime of strong magnetic fields corresponding to $q = 2$ and uniform ones, the method maintains a uniform second-order error bound. Numerical experiments verify the theoretical accuracy and the energy near preservation of the scheme. 

Last but not least, for $q = 1$ under strong magnetic fields, we observe that the numerical errors for specific test problems exhibit a more favourable dependence on $\varepsilon$ than predicted by existing theoretical analyses. We intend to investigate this phenomenon in future work. For the case $0 \leq q < 1$, it is necessary to explore alternative approaches to develop more efficient methods. This remains another direction for our future research.

\section*{Competing interests}
We declare that we have no conflict of interest.
\section*{Funding}
This work was supported partially by the  National Natural Science Foundation of China (Grant No. 12371403).

\bibliographystyle{elsarticle-num}      
\bibliography{tinging}

\end{document}